\UseRawInputEncoding
\documentclass[lettersize,journal]{IEEEtran}
\pdfoutput=1
\usepackage{amsmath,amssymb,bm,indentfirst}
\usepackage{mathrsfs}
\usepackage{algorithmic}
\usepackage{array}
\usepackage[caption=false,font=normalsize,labelfont=sf,textfont=sf]{subfig}
\usepackage{textcomp}
\usepackage{stfloats}
\usepackage{url}
\usepackage{verbatim}
\usepackage{graphicx}
\usepackage{color}
\usepackage{tablists}
\usepackage{authblk}
\usepackage[nocompress]{cite}
\usepackage{booktabs}
\usepackage{amsfonts,amsthm,amssymb,mathrsfs,bbding}
\usepackage{multirow}
\usepackage{setspace}
\usepackage{fancybox}
\usepackage[ruled,vlined,linesnumbered]{algorithm2e}
\usepackage{epstopdf}
\usepackage{xtab}
\def\BibTeX{{\rm B\kern-.05em{\sc i\kern-.025em b}\kern-.08em
    T\kern-.1667em\lower.7ex\hbox{E}\kern-.125emX}}
\usepackage{balance}
\usepackage{hyperref}
\newtheorem{lem}{Lemma}[section]
\newtheorem{thm}[lem]{Theorem}
\newtheorem{cor}[lem]{Corollary}

\newtheorem{rem}[lem]{Remark}
\newtheorem{defi}[lem]{Definition}

\date{}
\title{Reliability Evaluation of Generalized $K_4$-Hypercubes Based on Five Link Fault Patterns
\thanks{This work received partial support from
the National Natural Science Foundation of China (No. 12101528),
the Basic scientific research in universities of Xinjiang Uygur Autonomous Region (No. XJEDU2024P012),  and the Science and Technology Project of Xinjiang Uygur Autonomous Region (No. 2020D01C069). (\textit{Corresponding author: Mingzu Zhang}.)}}
\author{ Shuqian Cheng, Mingzu Zhang, Sun-Yuan Hsieh, \textit{Fellow, IEEE}, and Eddie Cheng  \thanks{Shuqian Cheng and  Mingzu Zhang are with the College of Mathematics and System Sciences, Xinjiang University,  Urumqi, Xinjiang, 830046,  China (e-mail:shuqcheng@163.com;mzuzhang@163.com).
 \par Sun-Yuan Hsieh is with the Department of Computer Science and Information Engineering, National Cheng Kung University, Tainan 70101, Taiwan (e-mail: hsiehsy@mail.ncku.edu.tw).
 \par Eddie Cheng is with the Department of Mathematics and Statistics, Oakland University, Rochester, MI 48309 USA (e-mail: echeng@oakland.edu).}}

\begin{document}
	\maketitle
\begin{abstract}
As the scale of data centers continues to grow, there is an increasing demand for interconnection networks to resist malicious attacks.
Hence, it is necessary to evaluate the reliability of networks under various fault patterns.
The family of generalized $K_4$-hypercubes serve as  interconnection networks of data centers, characterized by topological structures with exceptional properties.
The $h$-extra edge-connectivity $\lambda_h$, the $l$-super edge-connectivity $\lambda^l$, the $l$-average degree edge-connectivity $\overline{\lambda^l}$, the $l$-embedded edge-connectivity $\eta_l$ and the cyclic edge-connectivity $\lambda_c$ are vital parameters to accurately assess the reliability of interconnection networks. Let integer $n\geq3$.
This paper obtains the optimal solution of the edge isoperimetric problem and its explicit representation, which offers an upper bound of the $h$-extra edge-connectivity of an $n$-dimensional $K_4$-hypercube $H_n^4$.
As an application, we presents $\lambda_h(H_n^4)$ for $1\leq h\leq 2^{\lceil n/2  \rceil }$.
Moreover, for $2^{\lceil n/2\rceil+t}-g_t \le h\le2^{\lceil n/2\rceil+t}$, $g_t=\lceil(2^{2t+2+\gamma})/3\rceil$,
 $0\leq t \leq\lfloor n/2\rfloor-1 $,  $\gamma=0$ for even $n$ and $\gamma=1$ for odd $n$, $\lambda_h(H_n^4)$ is a constant $(\lfloor n/2\rfloor-t)2^{\lceil n/2\rceil+t}$.
The above lower and upper bounds of the integer $h$ are both sharp.
Furthermore, $\lambda^l(H_n^4)$, $\overline{\lambda^l}(H_n^4)$, $\lambda_{2^l}(H_n^4)$, and $\eta_l(H_n^4)$ share a common value $(n-l)2^l$ for  $2\leq l\leq n-1$, and we determines the values of $\lambda_c(H_n^4)$.

\end{abstract}
\section*{Notations}
\noindent
\begin{tabular}{ll}
$G(V, E)$ & A simple connected graph.\\
$E(G)$ & Edge set of graph $G$.\\
$V(G)$ & Vertex set of graph $G$.\\
$G\cong H$ & Graph $G$ and graph $H$ are isomorphic.\\
$G[X]$ & The  subgraph of $G$ induced by the set of \\& vertices $X$.\\
$[X,\overline{X}]_G$ & Set of edges of $G$ with exactly one endpoint in\\& $X \subseteq V (G)$.\\
$G-F$ & The subgraph obtained by removing a subset  $F$  of \\&$E(G)$ from $G$.\\
\end{tabular}

\noindent
\begin{tabular}{ll}
$ex_m(G)$ & The maximum sum of degree of subgraphs \\&induced by $m$ vertices of $G$.\\
$\xi_m(G)$ & The minimum cardinality of $[X,\,\overline{X}]$ with $|X|$\\&$=m\leq \lfloor|V (G)|/2\rfloor$, $G[X]$ and $G[\overline{X}]$ are both\\&connected.\\
$\lambda_h(G)$ & $h$-Extra edge-connectivity of $G$.\\
$\lambda^l(G)$ &  $l$-Super edge-connectivity of $G$. \\
$\overline{\lambda^l}(G)$ & $l$-Average degree edge-connectivity  of $G$. \\
$\eta_l(G)$ &  $l$-Embedded edge-connectivity of $G$.\\
$\lambda_c(G)$ &   Cyclic edge-connectivity of $G$.\\
$Q_n$ & An $n$-dimensional hypercube.\\
$\mathscr{H}_n^4$ & The family of generalized $K_4$-hypercubes.\\
$H_n^4$ & An $n$-dimensional $K_4$-hypercube.\\
$Q_{n,k}$ & An $n$-dimensional $(n,k)$-enhanced hypercube.\\
$\sum_{i=0}^{s}2^{t_i}$ & Binary decomposition of a positive  integer  with\\& $t_0>t_1>\cdots>t_s$.\\
$\lfloor x\rfloor$ $(\lceil x\rceil)$ & Floor (Ceiling) function of $x$.\\
$\gamma$ & $\gamma=0$ for even $n$   and $\gamma=1$ for odd $n$.\\
\end{tabular}

\section{Introduction}

\IEEEPARstart{I}{n} the present society, there is an overwhelming volume of big data, with approximately 294 billion emails sent  daily  (Stergiou et al. \cite{stergiou2020secure}, 2020).
With the popularization of the fifth-generation mobile communication technology (5G), and the rapid development of artificial intelligence,  the  demand for data centers is also increasing.
Large-scale data centers such as Internet of Things applications (IoT) are intertwined with big data \cite{wang2024machine}. After the product collects the data, it needs to be uploaded to the cloud for centralized processing.
Unfortunately, the incidence of malicious attacks targets  data centers, interconnection networks, software, applications, or data.
These attacks may exploit a variety of techniques and methods to violate the confidentiality, integrity, and availability of data center  networks (DCNs), posing threats to individuals, organizations, and government institutions. Some common types of cyber attacks include computer viruses, Distributed Denial-of-Service (DDoS) attacks, Structured Query Language (SQL) injection attacks, and so on. As modern society increasingly relies on digital technologies, cyber security issues are becoming more and more prominent.
To ensure that DCNs can still function effectively in the face of cyber attacks, a key method is to achieve fault tolerance through effective network topology.
Lv et al. \cite{Lv2023AHA} designed a server-centric DCN (called CSDC) with
excellent performance.
In detail, nodes in the network (such as processors) are represented as vertices and  links between nodes as edges.
The research on  fault tolerance of data centers focuses on the characterization of  the connectivity and reliability of networks. By studying the connectivity and reliability of the topology of networks, the ability of  data centers to withstand processor or link failures can be effectively assessed, providing important insights and guidance for the construction of efficient and stable DCNs.

\begin{figure}[t!]
	\begin{center} \includegraphics[width=1\linewidth]{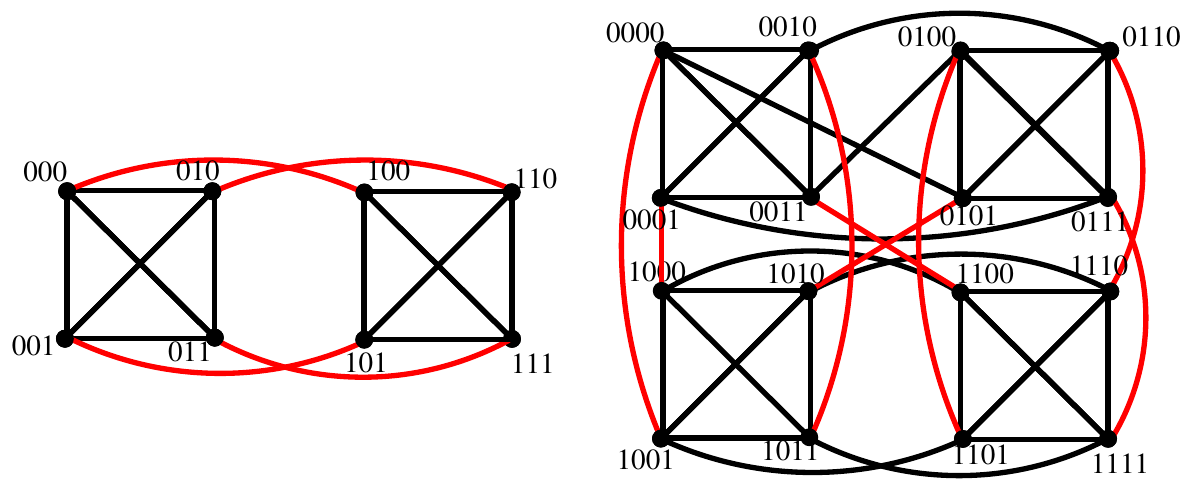}
		\caption{The graphs of $H_3^4$ $(Q_{3,2})$ and $H_4^4$.}
\label{The graphs of $H_3^4$ and $H_4^4$.}
	\end{center}
\end{figure}

The edge-connectivity $\lambda(G)$ of a connected graph $G$  \cite{menger1927allgemeinen}, if any, is a significant parameter that indicates the reliability  of  interconnection networks.
 It is defined as the minimum number of edges removed from the connected graph  $G$ to disconnect  $G$.
 The larger the edge-connectivity is, the more reliable an interconnection network is \cite{xu2001topological}.
 To  investigate  the properties of the remaining components, in 1983, Harary \cite{harary1983conditional} raised the concept of conditional edge-connectivity.
Assume that $\mathscr{P}$ is a property of a connected graph $G$.
 A $\mathscr{P}$-conditional  edge-cut of $G$, if exists, is a subset $F\subseteq E(G)$, such that $G-F$ is no longer connected and each connected component possesses the property $\mathscr{P}$.
The $\mathscr{P}$-conditional edge-connectivity of a connected graph $G$, if exists, is defined as the smallest size among all $\mathscr{P}$-conditional  edge-cuts of $G$, denoted by $\lambda(\mathscr{P},\,G)$.

Diverse properties raise the investigation of various $\mathscr{P}$-conditional edge-connectivities.
Let $h$ be a given integer.  With  $\mathscr{P}$ meaning that each component contains no less than $h$ vertices, $h$-extra edge-connectivity, denoted by $\lambda_{h}(G)$,
was first proposed by F\`{a}brega and Fiol \cite{fabrega1996extraconnectivity} in 1996.
With $\mathscr{P}$ representing that minimum degree at least $l$, the conditional edge-connectivity under this condition is called $l$-super edge-connectivity $\lambda^l(G)$,
 which was introduced by Li and Xu \cite{li2013edge}.
Lately, the $l$-average degree edge-connectivity $\overline{\lambda^l}(G)$ of $G$ was presented by Yang et al. \cite{yang2023average} in 2023, with $\mathscr{P}$ symbolizing that the average degree no less than $l$.
Moreover, if  $\mathscr{P}$ means that every vertex lies in an intact $l$-dimensional subnetwork of a recursive $n$-dimensional network $G$, the initial proposal of the related conditional edge-connectivity was made by
Yang and Wang \cite{yang2012conditional}, referred to the $l$-embedded edge-connectivity $\eta_l(G)$.
In 1880, the cyclic edge-connectivity  $\lambda_c(G)$, launched by Tait \cite{tait1880remarks},  is the minimum number of edges removed from $G$  and the residual graph has at least two components  holding cycles.
Let $\lambda^l(G),~\overline{\lambda^l}(G),~\lambda_{2^l}(G),~\eta_l(G),$ and $\lambda_c(G)$ be represented as $\lambda{(\mathscr{P}_i^l,~G)}$  for $i=1$ to $5$, respectively. These various connectivities are introduced  in order to depict the reliability of interconnected networks.

Many scholars studied the exact values of $\lambda_h(G)$ or $\lambda{(\mathscr{P}_i^l,~G)}$ for $1\leq i\leq 5$ of some classic interconnection networks, such as hypercube $Q_n$, folded hypercube $FQ_n$,  and enhanced hypercube $Q_{n,k}$.
Regarding the hypercube $Q_n$, scholars came up with the following conclusions over time.
Zhu and Xu  proved  $\lambda_2(Q_n)=2n-2$ for $n\geq4$ \cite{zhu2006restricted}.
Li and Yang  demonstrated $\lambda_h(Q_n)=nh-(\sum_{i=0}^s t_i 2^{t_i}+\sum_{i=0}^s 2i2^{t_i} )$ with $n\geq6$ and $1\leq h=\sum_{i=0}^{s}2^{t_i}\leq2^{\lfloor n/2 \rfloor}$\cite{li2013bounding}.
Zhang et al.   confirmed $\lambda_h(Q_n)=2^{n-1}$ with $n\geq4$ and $2^{n-1}/3<h\leq 2^{n-1}$\cite{zhang2016extra}.
Xu  bore out $\lambda^l(Q_n)=(n-l)2^l$ for $0\leq l\leq n-1$ in 2000\cite{xu2001topological}.
Recently, Liu et al. studied $g$-good $(r+1)$-component edge-connectivity  of $Q_n$\cite{liu2024reliable}.
Regarding the enhanced hypercube $Q_{n,k}$, Xu et al. studied   the reliability of the enhanced hypercube by $h$-extra connectivity\cite{xu2021reliability}.
Zhang et al.  proved $\lambda^1(Q_{n,~k})=2n$ and $\lambda^2(Q_{n,~k})=4n-4$\cite{Zhang.Liu}.
For $n\geq2$, $1\leq l,\,k\leq n-1$,
in addition, Liu et al.  determined $\lambda^l(Q_{n,~k})=\overline{\lambda^l}(Q_{n,~k})$
but for $l=n-2$  and $k=2$ is $2^{n-1}$ \cite{liu2023many}.
Moreover, Liu et al.  found the modified embedded edge-connectivity of $Q_{n,\,k}$, $\eta_l^\ast(Q_{n,\,k})=(n+1-l)2^l$, $n-k\geq l$ and $\eta_l^\ast(Q_{n,\,k})=(n-l)2^l$ with $n+1-k\leq l\leq n-1$, and they got $\lambda_c(Q_{n,\,k})$ with $n\geq 3$\cite{liu2023modified}.
Regarding the folded hypercube $FQ_n$,
Zhang et al. determined $\lambda^l(FQ_n)=(n-l+1)2^l$ for $0\leq l\leq n-1$\cite{zhang2022unified}.
Furthermore, Zhang et al.  proved  $\lambda_h(FQ_n)=(\lfloor n/2\rfloor-t+1)2^{\lceil n/2\rceil+t}$ for $n\geq 4$, $2^{\lceil n/2\rceil+t-1}\leq h\leq 2^{\lceil n/2\rceil+t}$, and $1\leq t\leq \lfloor n/2\rfloor-1$ in 2018\cite{zhang2018log}.
Recently, Yang  obtained  embedded edge-connectivity of $k$-ary $n$-cubes $\eta_l(Q_n^k)=2(n-l)k^l$, with odd $k\geq3$ and $0\leq l\leq n-1$\cite{yang2023embedded}.
Li et al.  determined the values of the $l$-super edge-connectivity $\lambda^l(Q_n^k)$ for $l\leq n$ and $k\geq3$, in 2024\cite{li2024reliability}.
Xu and Zhou used a unified approach to study the reliability of interconnection networks, including  $Q_n^3,\,Q_{n,k},\,AQ_n$, under different faulting links \cite{xu2023reliability}.
Our results for the class of the generalized $K_4$-hypercubes can be considered as an expansion of the above conclusions.

Due to the  ease  operation of hypercubes, the method of constructing  graphs by adding a matching from lower copies has gained significant attention.
Unlike bijective connection networks\cite{fan2003bc} obtained  repeatedly and randomly by adding a perfect matching between two lower copies from $K_2$, while generalized $K_4$-hypercubes are constructed by Brudnak and Lipt{\'a}k\cite{brudnak2021connectivity} repeatedly and randomly, adding a perfect matching between two lower copies from $K_4$.
The resulting graphs are called a class of $n$-dimensional $K_4$-hypercubes including a well-known interconnection network $(n,2)$-enhanced hypercube $Q_{n,2}$.
Brudnak and Lipt{\'a}k showed that the class of generalized hypercubes own very good connectivity properties. They proved that the generalized $K_4$-hypercubes are all maximally connected, and even for linearly many vertices are deleted, the remaining graph has a large connected component and  only a small number vertices in other components.
However, the research on the reliability evaluation of
generalized $K_4$-hypercubes under link fault patterns, as far as we know, has just begun.
For convenience, we rewrite the original notation of  the class of $n$-dimensional $K_4$-hypercubes  defined by Brudnak and Lipt{\'a}k.
An $n$-dimensional $K_4$-hypercube is denoted by $H_n^4$ and the family of $n$-dimensional $K_4$-hypercubes are denoted by $\mathscr{H}_n^4$. There is  a  $3$-dimensional $K_4$-hypercube $H_3^4$  and a $4$-dimensional $K_4$-hypercube $H_4^4$  illustrated in Fig.~\ref{The graphs of $H_3^4$ and $H_4^4$.}. And the perfect matching is represented by red lines.
In comparison to $Q_n$, the $n$-dimensional $K_4$-hypercubes exhibit many excellent measurements, including diameter, connectivity, and fault tolerance. With such remarkable properties, it is worthwhile to study the $n$-dimensional $K_4$-hypercubes.

The edge isoperimetric problem on a graph $G$ was proposed by Harper \cite{harper1964optimal} in 1964   to find a vertex subset in graph $G$ consisting  $m$ vertices, which $1\leq m<|V(G)|$, and satisfying the minimum number of edges that have to be deleted to separate this subset from graph $G$ among all  subsets of $m$ vertices separated from graph $G$. The minimum number is denoted as $\xi_m^e(G)$. Furthermore, it is closely related to the issue of reliability in parallel and distributed networks.
Denote by $G[X]$  the induced subgraph of the vertex set $X\subseteq V(G)$, and by
$[X,~\overline{X}]$  the set of edges of $G$, for every edge has one end vertex in the subset $X$ and the other end vertex in  the complement $\overline{X}=V (G)\backslash X$.
So we can express the explicit representation of the optimal solution of the edge isoperimetric problem: $\xi_m^e(G)=\min \{|[X, ~\overline{X}]|:|X|=m, 1\leq\ m\leq |V(G)|\}$.
Let
$
\xi_m(G)=\min\left\{|[X, ~\overline{X}]|:|X|=m \leq\lfloor|V(G)| / 2\rfloor,\right.\\
\left.G[X] \text { and  } G[\overline{X}]
\text{are both connected} \right\}.
$
For a $d$-regular graph $G$, $\xi_m(G)=dm-ex_m(G)$,
in which $e x_m(G)$ is  the maximum sum of degree among all subgraphs induced by $m$ vertices of $G$ for each $m \leq\lfloor|V(G)| / 2\rfloor$. Actually, if we can find $X_m^* \subseteq V(G),\left|X_m^*\right|=m$, with $e x_m(G)=2\left|E\left(G\left[X_m^*\right]\right)\right|$, and so that, both $G\left[X_m^*\right]$ and $G\left[\overline{X_m^*}\right]$ are connected. Thus,
$
\xi_m(G)=\left|\left[X_m^*, \overline{X_m^*}\right]\right|=d m-e x_m(G)=d m-2 |E\left(G\left[X_m^*\right]\right)|$.
So sometimes the problem of finding the maximum number of edges in a subgraph $G[X]$ induced from a graph $G$ with $m$ vertices is referred to as the edge isoperimetric problem \cite{bezrukov2018new}.
It is considered $\lambda_h$-optimal if $\lambda_h(G)=\xi_h(G)$; otherwise, it is not $\lambda_h$-optimal.
In general, for any $1\leq m \leq \lfloor|V(G)| / 2\rfloor$, $\xi_m(G)$ and  $\xi_m^e(G)$ are not necessarily equal pairwise. However, for the $n$-dimensional $K_4$ hypercubes, $\xi_m(G)=\xi_m^e(G)$ holds. In the following, we will use $\xi_m(G)$ to represent them uniformly.

The main contribution of this paper is that we study the reliability of the family of generalized $K_4$-hypercubes based on five link fault patterns.
Let  $\gamma=0$ for even $n$   and $\gamma=1$ for odd  $n$. It will be applied throughout the entire paper and will not be reiterated afterward.
In this paper, we obtain  the explicit representation of the optimal solution of the edge isoperimetric problem for our networks, which offers an upper bound of the $h$-extra edge-connectivity of an $n$-dimensional generalized $K_4$-hypercube $H_n^4$. As an application, we analyze the reliability of  $n$-dimensional $K_4$-hypercubes based on five link fault patterns.
Let $n,\,h,\,l$ be integers with $n\geq3$.
We derive the values of $\lambda_h(H_n^4)$ for  $1\leq h\leq 2^{\lceil n/2  \rceil }$.
Moreover, we give the exact values   of $\lambda_h(H_n^4)$  for  some exponentially large enough $h$ with the lower and upper bounds of each integer $h$ in $\lambda_h(H_n^4)$ are sharp.
Furthermore, we draw a conclusion that $\lambda(\mathscr{P}_i^l,~H_n^4)$ share a common value for $2\leq l\leq n-1$ and $ 1\leq i\leq4$. In addition, we derive the cyclic edge-connectivity $\lambda(\mathscr{P}_5^l,~H_n^4)=4n-8$ if $n=3,\,4$ and $\lambda(\mathscr{P}_5^l,~H_n^4)=3n-3$ if $n\geq5$.

The rest of this paper is organized as follows.
In section \ref{sec2}, the related definitions and the basic structural properties of  the $n$-dimensional $K_4$-hypercubes are introduced.
In section \ref{sec3}, the structure of the densest edges of the subgraph induced by $H_n^4$ is given.
In section \ref{sec4}, some properties  of
 the optimal solution of the edge isoperimetric problem
 are demonstrated.
In section \ref{sec5}, by using the expression of
 the optimal solution of the edge isoperimetric problem,
  the main three results  of $n$-dimensional $K_4$-hypercubes under five link fault patterns are proven.
In section \ref{sec6}, the conclusion is  provided.
%%%%%%%%%%%%%%%%%%%%%%%%%%%%%%%%%%%%%%%%%%%%%%%%%%%%%%%%%%%%%%%%%%%%%%%%%%%%%%%%%%%%%%%%%%%%%%%%%%%%%%%%%%%%%
\begin{thm}\label{thm1}
Let $n$ and $h$ be integers, with $n\geq3$ and $1\leq h \leq 2^{\lceil n/2  \rceil }-2-\gamma$.
Then
$$\lambda_{h}(H_{n}^{4})=\xi_{h}(H_{n}^{4}).$$
\end{thm}

\begin{thm}\label{thm2}
For integer $n\geq4  $ and any  integer $h$ in the range $2^{\lceil n/2\rceil+t}-g_t \le h\le2^{\lceil n/2\rceil+t}$ with $t=0,1, \cdots,\lfloor n / 2\rfloor -1$, $g_t=\lceil\left(2^{2 t+2+\gamma} \right)/ 3\rceil$,
then
$$\lambda_{h}(H_{n}^{4})=\xi_{2^{\lceil n/2\rceil+t}}(H_{n}^{4})=(\lfloor n/2\rfloor-t)2^{\lceil n/2\rceil+t}.$$
\end{thm}

\begin{thm}\label{thm3}
Let $n,\,l,\,i$ be three positive integers  satisfying $n\geq3$, $2\leq l\leq n-1$, and $ 1\leq i\leq5$. Then\\
$(1).$ $\lambda(\mathscr{P}_i^l,~H_n^4)=(n-l)2^l$
for $1\leq i\leq4;$
\\
$(2).~$ $\lambda(\mathscr{P}_5^l,~H_n^4)=
4n-8$ if $n=3,~4;$ $\lambda(\mathscr{P}_5^l,~H_n^4)=3n-3$ if $n\geq 5$.
\end{thm}

\section{Preliminaries}\label{sec2}
%%%%%%%%%%%%%%%%%%%%%%%%%%%%%%%%%%%%%%%%%%%%%%%%%%%%%%%%%%%%%%%%%%%%%%%%%%%%%%%%%%%%%%%%%%%%%%%%%%%%%%%%%%%%%

%%%%%%%%%%%%%%%%%%%%%%%%%%%%%%%%%%%%%%%%%%%%%%%%%%%%%%%%%%%%%%%%%%%%%%%%%%%%%%%%%%%%%%%%%%%%%%%%%%%%%%%%%%%%%
\par
In this paper, we refer to \cite{descartes1976ja} for concepts not mentioned.
An interconnection network can be topologically depicted as a graph $G=(V(G),\,E(G))$, in which $V(G)$ represents the set of vertices and $E(G)$ represents the set of edges.
The cardinalities of the vertex and edge sets are denoted by $|V(G)|$ and $|E(G)|$, respectively.

Firstly, the definitions of the $n$-dimensional hypercube, the $(n,k)$-enhanced hypercube, and the generalized $K_4$-hypercubes are stated as follows.

\begin{defi}\label{defi1}\textup{\cite{sullivan1977large}} Let $n$ be  an integer, with $n \geq 1$.
The $n$-dimensional hypercube is denoted by $Q_n$ in which the vertex set is the set of all $n$-bit binary strings and is represented by $V\left(Q_n\right)=\left\{y_n y_{n-1} \ldots y_1: y_i \in\{0,1\}, 1 \leq\right.$ $i \leq n\}$ .
Two vertices $x=x_n x_{n-1} \cdots x_2 x_1$ and $y=y_n y_{n-1} \cdots y_2 y_1$ of $Q_n$ are adjacent if and only if they differ in exactly one position.
\end{defi}
For any vertices $u=u_n u_{n-1} \cdots u_2 u_1$ and $v=v_n v_{n-1} \cdots v_2 v_1$, the edge $uv$ is called $k$-complementary edges $(1 \leq k \leq n-1)$ if and only if $v_i=u_i$ for $n-k+1<i \leq n$, and $v_j=\overline{u_j}$ for $1 \leq j \leq n-k+1$.
\begin{defi}\textup{\cite{tzeng1991enhanced}}
Assume that  $k$ and $n$ are two positive integers in which $1\leq k\leq n-1$. The $(n, k)$-enhanced hypercube  $Q_{n,k}$  is a supergraph of $Q_n$ by adding all $k$-complementary edges: $E(Q_{n,k})=E(Q_n)\bigcup\big\{(y_ny_{n-1}\cdots y_{n-k+2} y_{n-k+1}y_{n-k}\cdots y_2y_1,\,y_n y_{n-1}$
$\cdots y_{n-k+2}\overline{y_{n-k+1}}~\overline{y_{n-k}}\cdots \overline{y_2}~ \overline{y_1}): y_j = 0$ or $1,\,1\leq j \leq n\big\}$.
\label{defiQn,k}
\end{defi}

\begin{defi}\label{defi2}\textup{\cite{brudnak2021connectivity}}
The family of $n$-dimensional $K_4$-hypercubes, denoted by $\mathscr{H}_n^4$, is defined  for $n\geq2$ recursively: the only $2$-dimensional $K_4$-hypercube is $K_4$, i.e, $\mathscr{H}_2^4$ contains only $\{K_4\}$. And if $n\geq3$, the graphs in $\mathscr{H}_n^4$ are exactly the graphs, up to isomorphism, that can be  constructed by taking two copies $($ which may not isomorphic$)$ of graphs from $\mathscr{H}_{n-1}^4$ and adding a perfect matching, arbitrarily, between them.
\end{defi}
For $k=n-1$, the $(n,n-1)$-enhanced hypercube $Q_{n,n-1}$ is obtained from $Q_n$ by adding $(n-1)$-complementary edges.
Note that $Q_{n,n-1}$ is contained in $\mathscr{H}_n^4$.

  For $n\!\!\!\!\geq\!\!\!\!2$, by the definition of generalized $K_4$-hypercubes, denote by $X_n X_{n-1}X_{n-2}\cdots X_{2}X_{1}\!=\{x_{n} x_{n-1} x_{n-2} \cdots x_2x_1|x_i\in \{0,1\}, 1\leq i\leq n\}$ the vertex set $V(H_n^4)$.
Let $V_0=0 X_{n-1} X_{n-2} \cdots X_2X_1$ and $V_1=1 X_{n-1} X_{n-2} \cdots X_2X_1$. So $H_n^4$ is obtained by arbitrarily adding a perfect matching between $V_0$ and $V_1$.
Furthermore, denote by  $y_{n}y_{n-1} \cdots y_{k+1}X_k X_{k-1} \cdots X_2 X_1$ the set of $y_{n}y_{n-1} \cdots y_{k+1}x_a \cdots x_2x_1: x_i \in \{0,1\},1\leq i\leq k$ with $y_j$ is fixed for every $a+1\leq j\leq n$.
One can see that  $H_n^4[y_{n}y_{n-1} \cdots y_{k+1}X_k  X_{k-1} \cdots X_2 X_1]$ is the subgraph, induced by $y_{n}y_{n-1} \cdots y_{k+1}X_k X_{k-1} \cdots X_2 X_1$.
Then  we can denote by $y_{n}y_{n-1} \cdots y_{k+1}X_k X_{k-1}\cdots X_2 X_1$ the $k$-dimensional sub-$H_n^4$ in an $n$-dimensional $K_4$-hypercube.
Hence, $\underbrace{00\cdots 0}_{n-t_0-1}X_{t_0+1} X_{t_0} \cdots X_2X_1$ can be represented as $\underbrace{00\cdots 00}_{n-t_0}X_{t_0}X_{t_0+1} X_{t_0} \cdots X_2X_1$ and $\underbrace{00\cdots 01}_{n-t_0}X_{t_0}X_{t_0+1} X_{t_0} \cdots X_2X_1$.

\begin{figure*}[tb]
\centering
\includegraphics[width=1\linewidth]{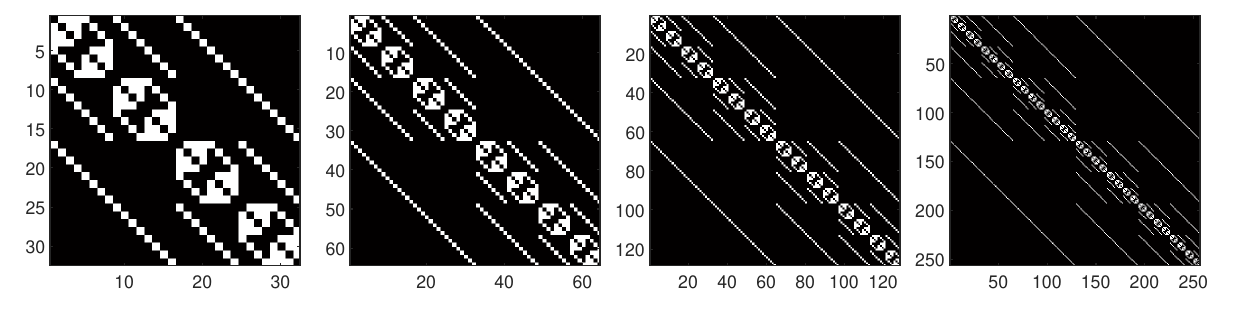}
\caption{Bitmaps of adjacency matrix of $H_n^4$ for $3\leq n\leq8$.}
\label{The bitmaps of adjacency matrix}
\end{figure*}
Harper \cite{harper1964optimal}, Li and Yang\cite{li2013bounding} independently obtained the exact value of $ex_m(Q_n)$.
\begin{lem}\label{lem2.3}\cite{harper1964optimal,li2013bounding}
Given an  integer  $m=\sum_{i=0}^{s}2^{t_i}\leq2^n$, the following results are valid:\\
$(1).$ $ex_m(Q_n)\!\!=2|E(Q_n[L_m^n])|\!\!=\sum_{i=0}^{s}t_i2^{t_i}+\sum_{i=0}^{s}2i2^{t_i}$
and $\xi_m(Q_n)=nm-ex_m(Q_n)$;\\
$(2).$ $ex_{m+1}(Q_n)-ex_m(Q_n)=2s+2;$\\
$(3).$ $ex_m(Q_n)=ex_{m-2^{t_0}}(Q_n)+ex_{2^{t_0}}(Q_n)+2(m-2^{t_0})$.
\end{lem}
Note that $H_n^4$ can be split into two parts recursively and each has $2^{n-1}$ vertices with a perfect matching between the two parts. It
is an $(n+1)$-regular graph with $2^n$ vertices and $(n+1)2^{n-1}$ edges. So
\begin{equation}
\label{equa1}
\xi_m(H_n^4)=(n+1)m-ex_m(H_n^4).
\end{equation}
By the definition of the $h$-extra edge-connectivity of a graph $G$, for all $1 \leq h \leq2^{n-1}$,
 we have
\begin{equation}
\label{equa2}
\lambda_h(H_n^4)=\min \left\{\xi_m(H_n^4): h \leq m \leq 2^{n-1}\right\} .
\end{equation}

With the increase of $n$,
the scale of the $H_n^4$ grows exponentially, and its topological structure becomes more and more complex. Therefore, the bitmaps of the adjacency matrix of $H_n^4$ distinctly illustrate the adjacent relationship between its vertices. The bitmaps of the adjacency matrix of $H_n^4$ for $3\leq n\leq8$ are  presented in Fig.~\ref{The bitmaps of adjacency matrix}. In these bitmaps, a white pixel at position $(u, v)$ corresponds to an edge between vertex $u$ and vertex $v$.

Although these five types of edge-connectivities have different failure scenarios, they share a common characteristic: removing the minimum edge-cut results in two  components in the remaining graph.

\begin{lem}\label{lem2.4}
If $F$ is any minimum  $h$-extra edge-cut or $\mathscr{P}_i^l$ edge-cut of $H_n^4$ for $1\leq i\leq 5$ and $2\leq l\leq n-1$, then $H_n^4-F$  has exactly two components.
\end{lem}
\begin{proof}
We prove by contradiction.  Assume that $F$ is the minimum $h$-extra edge-cut or $\mathscr{P}_i^l$  edge-cut of $H_n^4$, such that $H_n^4-F$ has $p$ components $S_1, S_2, \cdots, S_p$ with at least $h$ vertices or satisfying $\mathscr{P}_i^l$,  $p \geq 3$. Because $H_n^4$ is connected, there must exist integers $i, j$ with $\left[V\left(S_i\right), V\left(S_j\right)\right] \neq \emptyset,\left|F_1\right|= \left|F \backslash\left[V\left(S_i\right), V\left(S_j\right)\right]\right|<\left|F\right|$. Thus, $F_1$ is also a minimum $h$-extra edge-cut of $H_n^4$, which contradicts the minimality of $F$. Thus, $H_n^4-F$ has exactly two components.
\end{proof}

\section{The structure of the densest edges of subgraph induced by $H_n^4$}
\label{sec3}
Moving forward, we  explore the expression of $ex_m(H_n^4)$.
We present the unique binary representation of $m$ and the construction of vertex set $L_m^n$ for the sake of the expression of $ex_m(H_n^4)$.
Consider an  integer $m$ with $1\leq m\leq2^{n}$. Assume that  $m=\sum_{i=0}^s2^{t_i}$ is the decomposition of $m$,  $t_0= \lfloor \log _2 m \rfloor $, $t_i=\lfloor \log_2(m-\sum_{r=0}^{i-1}2^{t_r})\rfloor$, $i=1,2,\cdots,s$, and $t_0>t_1>\cdots>t_s$.
Let $S_m=\{0,1,\cdots,m-1\}($under decimal representation$)$, $|S_m|=m$, $m\leq2^{n-1}$, and the corresponding set $L_m^n$ represented by $n-$binary strings.
 Denote by $L_m^n$  the subset of $V(H_n^4)$ and by $H_n^4[L_m^n]$  the subgraph induced by $L_m^n$ in $H_n^4$.
Both $H_n^4[L_m^n]$ and $H_n^4[\overline{L_m^n}]$ are connected.
Due to the fact that the spanning subgraph of a graph $H_n^4$ obtained by $(n-1)$-complementary edges deletions only  belongs to bijective connection networks $\mathscr{B}_n$ and both $B_n[L_m^n]$ and $B_n[\overline{L_m^n}]$ are connected, it follows that $H_n^4[L_m^n]$ and $H_n^4[\overline{L_m^n}]$ are also connected.

Given a positive integer $m$ with $1\leq m=4p+q=\sum_{i=0}^{s}2^{t_i}\leq2^n$, the function $f(m)$ is defined as:\\
\begin{math}
\begin{aligned}
f(m)
=\begin{cases}
\sum_{i=0}^s t_{i}2^{t_i}+\sum_{i=0}^s2 i \cdot 2^{t_i}+4p&\text{if }0\leq q\leq 2;\\
\sum_{i=0}^s t_{i}2^{t_i}+\sum_{i=0}^s2i \cdot 2^{t_i}+4p+2&\text{if }q=3
\end{cases}
\end{aligned}\end{math}\\
for $p=\lfloor m/4\rfloor$, $q=m-4\lfloor m/4\rfloor$.
In this section, we prove $ex_m(H_n^4)= f(m)$. The following lemma gives an upper bound for $ex_m(H_n^4)$.
\begin{lem}\label{lem3.1}
$ex_m(H_n^4)\leq f(m)$.
\end{lem}
\begin{proof}
By induction on $n$.
Suppose that $X\subseteq V(H_n^4)$ and $ex_m(H_n^4)=2|E(H_n^4[X])|$.
Let $A_0$ and $A_1$ be induced subgraphs $H_n^4[V_0]$ and $H_n^4[V_1]$.
 %$\left( \right.$Both $V_0$ and $V_1$ belong to $V(H_n^4)$ and satisfy $V(H_n^4)=V_0\cup V_1$, $V_0\neq \emptyset$, $V_1\neq \emptyset$, $V_0\cap V_1=\emptyset$, and $H_n^4[V_0]$, $H_n^4[V_1]\in \mathscr {H}_{n-1}^4$; and the edges between $A_0$ and $A_1$ form a perfect matching of $\left.H_n^4.\right)$
For $n =2,~3$, we can  derive the equation. Thus  suppose $n\geq4,\,m=4p+q,\,0\leq q<4$.
 Let $X_0$ and $X_1$ be denoted by $X\cap V_0$ and $X\cap  V_1$.
 Without loss of generality,  assume that $ m_i = |X_i|$, $i = 0,1$. Let $m_0 = \sum_{i=0}^{s_0}2^{t_{0i}}\leq   m_1=\sum_{i=0}^{s_1}  2^{t_{1i}}$.  If $m_0 =0$, by the induction hypothesis, then $ex_m(H_n^4)\leq f (m)$, and so assume $m_0>0$ from now on.
\par\textbf{Case 1.} Suppose $s = 0$.
 \par\textbf{Subcase 1.1.} If $t_{00} = t_0-1$,
checking $ex_m(H_n^4)\leq f(m)$
 for $t_0=1,\ 2$ can be conducted directly.
 Note that there are at most $m_0$ edges between $A_0\left[X_0\right]$ and $A_1\left[X_1\right]$. The representations $ex_{m_0}(H_n^4)$ and $ex_{m_1}(H_n^4)$ are restricted in $A_0$ and $A_1$.
 Omit  similar contexts  from this moment.
 Given $t_0\geq3$, all three variables $m$, $m_0$ and $m_1$ are divisible by $4$, resulting in $ex_m(H_n^4)=2|E(H_n^4[X])|\leq2|E(H_n^4[X_0])| + 2|E(H_n^4[X_1])| + 2m_0\leq ex_{m_0}(H_n^4)+ex_{m_1}(H_n^4) + 2m_0 \leq f(m_0) + f(m_1) + 2m_0 = (t_0-1)2^{t_0-1} +m_0+ (t_0-1)2^{t_0-1}+m_0+2m_0
 %=2t_0\cdot2^{t_0-1}-2\cdot2^{t_0-1}+2^{t_0+1}
% =(t_0-1)2^{t_0}+2\cdot2^{t_0}
 = t_0\cdot2^{t_0}+2^{t_0}= f (m)$.
 \par\textbf{Subcase 1.2.} If $t_{00}<t_0-1$.
Note that $m_0+m_1=2^{t_0}$.  It can be seen that $t_{0 s_0}=t_{1 s_1}$ and
\begin{equation}
\begin{split}
&\left\{t_{00}, t_{01}, \cdots, t_{0\left(s_0-1\right)}, t_{10}, \cdots, t_{1\left(s_1-1\right)}\right\}\\
&=\left\{t_0-1, t_0-2, \cdots, t_{0 s_0}+1\right\}.
\end{split}
\label{equa3}
\end{equation}
Actually, it can be derived recursively. Note that $t_{10}=t_0-1$,  $m_1>m_0$, then $m_1-2^{t_0-1}+m_0=2^{t_0-1}$. Let $m_1^{\prime}=m_1-2^{t_0-1}, m_0^{\prime}=m_0$. If $m_1^{\prime}>m_0^{\prime}$, then $t_{11}=t_0-2$; If $m_0^{\prime}>m_1^{\prime}$, then $t_{00}=t_0-2$; If $m_0^{\prime}=m_1^{\prime}$, then $t_{11}=t_{00}=t_0-2$. Since $2^{t_0}$ is finite, one can find that $t_{0 s_0}=t_{1 s_1}$. Clearly, $\left|\left\{t_{00}, t_{01}, \cdots, t_{0\left(s_0-1\right)}, t_{10}, \cdots, t_{1\left(s_1-1\right)}\right\}\right|=$ $\left|\left\{t_0-1, t_0-2, \cdots, t_{0 s_0}+1\right\}\right|=s_0+s_1$.

We put the two expressions $\left(m_0=\sum_{i=0}^{s_0} 2^{t_{0 i}}\right.$ and $\left.m_1=\sum_{i=0}^{s_1} 2^{t_{1 i}}\right)$ together as follows:
\begin{equation}
\begin{split}
m&=m_0+m_1  =2^{t_0}
=\sum\nolimits_{l=t_{0s_{0}}+1}^{t_0-1} 2^l+2^{t_{0s_0}}+2^{t_{0s_1}} \\
&=2^{t_0-1}+\cdots+2^{t_0-j}+\cdots
+2^{t_{0s_0}+1}+2^{t_{0s_0}}+2^{t_{0s_0}}.
\end{split}
\label{equa4}
\end{equation}
(Observe that (2) contains exactly $\left(s_0+1\right)+\left(s_1+1\right)$ terms and $4|m$.) From the function $f(m)$, we can calculate
\begin{equation}\label{equa5}
\begin{split}
&f\left(m\right)\\
%&=
%f\left(2^{t_0-1}+\cdots+2^{t_0-j}+\cdots
%+2^{t_{0s_0}+1}+2^{t_{0s_0}}+2^{t_{0s_0}}\right)\\
&= \left(t_0-1\right) 2^{t_0-1}+2^{t_0-1}\\
&+\left[\left(t_0-2\right) 2^{t_0-2}+2 \cdot1\cdot 2^{t_0-2}+2^{t_0-2}\right]\\
&+\cdots+\left[\left(t_0-i\right) 2^{t_0-i}+2 \cdot(i-1) \cdot 2^{t_0-i}+2^{t_0-i}\right] \\
&+\cdots+\left[\left(t_{0 s_0}\right) 2^{t_{0 s_0}}+2 \cdot\left(t_0-t_{0 s_0}-1\right) \cdot 2^{t_{0 s_0}}+2^{t_{0 s_0}}\right]\\
&+\left[\left(t_{0 s_0}\right) 2^{t_{0 s_0}}+2 \cdot\left(t_0-t_{0 s_0}\right) \cdot 2^{t_{0 s_0}}+2^{t_{0 s_0}}\right].
\end{split}
\end{equation}

By (\ref{equa3}), for any term $2^{t_{0 i}}$ in $m_0=\sum_{i=0}^{s_0} 2^{t_{0 i}}$, there is a term $2^{t_0-j}$ in  (\ref{equa4}) such that $t_0-j=t_{0i}$, and hence there is a term $\left[\left(t_0-j\right) 2^{t_0-j}+2 \cdot(j-1) \cdot 2^{t_0-j}+2^{t_0-j}\right]$ in  (\ref{equa5}) corresponding to the term.
Hence, $\left(t_0-j\right) 2^{t_0-j}+2 \cdot(j-1) \cdot 2^{t_0-j}+2^{t_0-j}=t_{0 i} 2^{t_{0 i}}+2 \cdot i \cdot 2^{t_{0 i}}+2 \cdot(j-1-i) \cdot 2^{t_{0 i}}+2^{t_{0i}}$. Due to $m_1>m_0$, it follows that $j-2 \geq i$.
%(for example, $t_0-2 \geq t_{00}$ and hence $\left.j=2, i=0, j-2=0 \geq i=0\right)$.
This indicates $\sum_{i=0}^{s_0} 2 \cdot(j-1-i) \cdot 2^{t_{0 i}} \geq \sum_{i=0}^{s_0} 2 \cdot 1\cdot2^{t_{0 i}}=2 m_0$.
Assume that $m_0=4p_0+q_0$. % Given that
%\begin{math}
%\begin{aligned}
%f(m_0)=\begin{cases}
%\sum_{i=0}^{s_0} t_{i}2^{t_i}+\sum_{i=0}^{s_0}2 i 2^{t_i}+m_0-q_0&\text{if }q_0\neq3;\\
%\sum_{i=0}^{s_0} t_{i}2^{t_i}+\sum_{i=0}^{s_0}2 i %2^{t_i}+m_0-1&\text{if }q_0=3,
%\end{cases}
%\end{aligned}\end{math}\\
We can rewrite it as $f(m_0)=\sum_{i=0}^{s_0} t_{i}2^{t_i}+\sum_{i=0}^{s_0}2 \cdot i \cdot 2^{t_i}+m_0-a$, $a=0,\,1,\,2$.
As a result, for $m_0$, there is a term $f(m_0)+\sum_{i=0}^{s_0} 2 \cdot(j-1-i) \cdot 2^{t_{0 i}}+a$ in $(3)$.
Thus, $f\left(m_0\right)+f\left(m_1\right)+2 m_0 \leq f(m)$ for $m=2^{t_0}=m_0+m_1$ and $m_1>m_0>0$.

By induction, we have $e x_{m_0}(H_n^4)+e x_{m_1}(H_n^4)\leq f\left(m_0\right)+f\left(m_1\right)$. Since $f\left(m_0\right)+f\left(m_1\right)+2 m_0 \leq f(m)$, we have $e x_m (H_n^4)\leq$ $e x_{m_0}(H_n^4)+e x_{m_1}(H_n^4)+2 m_0\leq f\left(m_0\right)+f\left(m_1\right)+2 m_0 \leq f(m)$ for $m_0 \geq 0$.
\par\textbf{Case 2.} Suppose $s>0$.
We demonstrate this case, recursively, and outline the following two operations:
\par$\left(O_1\right)$ If $t_{10}>t_{00}$, consider $T^1 \subset X_1$  with $\left|T^1\right|=2^{t_{10}}$.  Assume that $X^1=X \backslash T^1, X_0^1=X_0, X_1^1=X_1 \backslash T^1$, and $m^1=m-\left|T^1\right|, m_0^1=m_0, m_1^1=m_1-\left|T^1\right|$.

Firstly, we consider the following two subcases: $2^{t_{10}}$ is a term of the expression $m=\sum_{i=0}^s 2^{t_i}$, and the other.

If $2^{t_{10}}$ is a term of the expression $m=\sum_{i=0}^s 2^{t_i}$, then either $t_{10}=t_0$ or $t_{10}=t_1=t_0-1$. (Note that $m_1 \geq m/2$ and $s>0$. Then $m_1>2^{t_0-1}$ and hence $t_{10} \geq t_0-1$.)
Due to the fact that $m^1=m-2^{t_{10}}$, the remainders of $m$ and $m^1$ upon being divided by $4$ are equal. In light of this,
if $t_{10}=t_0$, then we can assume that $f\left(m^1\right)=\sum_{i=1}^s t_i 2^{t_i}+\sum_{i=1}^s 2 \cdot(i-1) \cdot 2^{t_i}+m^1-a$ and  $f(m)=\sum_{i=0}^s t_i 2^{t_i}+\sum_{i=0}^s 2 \cdot i \cdot 2^{t_i}+m-a$, with $a=0,\ 1,\ 2$. Accordingly, $f\left(m^1\right)$ is obtained from $f(m)$ by removing $t_0 2^{t_0}+2\left(2^{t_1}+\cdots+2^{t_s}\right)+2^{t_0}=t_0 2^{t_0}+2 m^1+2^{t_0}$, that is, $f(m)=f\left(m^1\right)+2 m^1+t_0 2^{t_{10}}+2^{t_0}$.
If $t_{10}=t_1=t_0-1$, then $f\left(m^1\right)$ is obtained from $f(m)$ by removing $t_1 2^{t_1}+2\left(2^{t_1}+\cdots+2^{t_s}\right)+2^{t_1}$, that is, $f(m) \geq f\left(m^1\right)+2 \cdot 2^{t_{10}}+t_{10} 2^{t_{10}}+2^{t_{10}}$.

If $2^{t_{10}}$ is not a term of the expression $m=\sum_{i=0}^s 2^{t_i}$, then $t_{10}=t_0-1>t_1$.
It is worth noting that $f(m)=$ $\left(t_{10}+1\right) 2^{t_{10}+1}+\sum_{i=1}^s t_i 2^{t_i}+\sum_{i=1}^s 2  i \cdot 2^{t_i}+m-a$ and $f\left(m^1\right)=t_{10} 2^{t_{10}}+\sum_{i=1}^s t_i 2^{t_i}+\sum_{i=1}^s 2  i \cdot 2^{t_i}+m^1-a$ for $a=0,\ 1,\ 2$. It means that  $f(m)-f\left(m^1\right)=$ $\left(t_{10}+1\right) 2^{t_{10}+1}-t_{10} 2^{t_{10}}+2^{t_{10}}=t_{10} 2^{t_{10}}+2^{t_{10}+1}+2^{t_{10}}$. As a consequence,
 $f(m) \geq f\left(m^1\right)+2 \min \left\{m^1, 2^{t_{10}}\right\}+t_{10} 2^{t_{10}}+2^{t_{10}}$.

Because of $2\left|E\left(H_n^4[X]\right)\right| \leq 2\left|E\left(H_n^4\left[X-T^1\right]\right)\right|+2 \min \left\{m^1, 2^{t_{10}}\right\}+t_{10} 2^{t_{10}}+2^{t_{10}}$, combining this with $f(m) \geq f\left(m^1\right)+$ $2 \min \left\{m^1, 2^{t_{10}}\right\}+t_{10} 2^{t_{10}}+2^{t_{10}}$, it can be seen that $2\left|E\left(H_n^4[X]\right)\right| \leq f(m)$ follows from $2\left|E\left(H_n^4\left[X-T^1\right]\right)\right|=2\left|E\left(H_n^4\left[X^1\right]\right)\right| \leq f\left(m^1\right)$.
\par$\left(O_2\right)$ If $t_{10}=t_{00}$, assume that $T_i^1 \subset X_i$ with $\left|T_i^1\right|=2^{t_{10}}=2^{t_{00}}$. Define $T^1=T_0^1 \cup T_1^1, X^1=X \backslash T^1, X_i^1=X_i \backslash T_i$, and $m^1=m-\left|T_0^1\right|-\left|T_1^1\right|, m_i^1=m_i-\left|T_i^1\right|$.

In the same manner, we catch on to the fact that $f(m) \geq f\left(m^1\right)+2 \min \left\{m^1, 2^{t_{10}+1}\right\}+t_{10+1}2^{t_{10}+1}+ 2^{t_{10}+1}$ and $2\left|E\left(H_n^4[X]\right)\right| \leq 2\left|E\left(H_n^4\left[X-\left(T_0 \cup T_1\right)\right]\right)\right|+$ $2 \min \left\{m^1, 2^{t_{10}+1}\right\}+t_{10+1}2^{t_{10}+1}+ 2^{t_{10}+1}$. As a result, we have $2\left|E\left(H_n^4[X]\right)\right| \leq f(m)$ by $2\left|E\left(H_n^4\left[X-T^1\right]\right)\right|=2\left|E\left(H_n^4\left[X^1\right]\right)\right| \leq f\left(m^1\right)$.

Now,
we repeat $\left(O_1\right)$ and/or $\left(O_2\right) i$ times $\left(i \leq \min \left\{s_0, s_1\right\}\right)$ such that either $X^i \subset V\left(A_0\right)$ (or $V\left(A_1\right)$ ), or $m^i=2^l$ for some integer $l$. Therefore, $2\left|E\left(H_n^4\left[X^i\right]\right)\right| \leq f\left(m^i\right)$ is  derived by induction, or by Case 1.
 The proof is  completed.
\end{proof}

\begin{thm}\label{thmex_m(H_n^4)=f(m)}
$ex_m(H_n^4)=f(m)$.
\end{thm}
\begin{proof}
Based on the previous lemma, the only requirement is to demonstrate that
 $e x_m(H_n^4) \geq f(m)$. It is sufficient to show that there is a subgraph $H_n^4\left[L_m^n\right]$ in $H_n^4$ satisfying $\left|V\left(H_n^4\left[L_m^n\right]\right)\right|=m,~\left|E\left(H_n^4\left[L_m^n\right]\right)\right|=f(m)/2$.
 Let $V(H^0)=0 \cdots 0 X_{t_0}\cdots X_{1}=\{0\cdots0x_{t_0}\cdots x_1:x_i\in\{0,1\},~1\leq i\leq t_0\}$.
 Suppose, furthermore, that $H^i$ is taken from a $(t_{i-1}+1)$-dimensional subcube gained from $H^{i-1}$ by changing the 0 of $\left(t_{i-1}+1\right)$th-coordinate of $H^{i-1}$ to 1.
 We choose $H_n^4\left[L_m^n\right]$ in the following manner.
\par Take $t_i$-dimensional sub-$H_n^4$ graphs   as follows ($t_i\geq2$):\\
$
~~~~~~H^0:  0 \cdots 0  \underbrace{X_{t_0} \cdots X_1}_{t_0}
\\\left(t_0\right.\text{-dimensional sub-$H_n^4$ graph) } \\
~~~~~~H^1:  0 \cdots 01\underbrace{0 \cdots 0 \underbrace{X_{t_1 } \cdots X_{1}}_{t_1}}_{t_0}
\\\text { (take a } t_1 \text {-dimensional sub-$H_n^4$ graph }\\
\text {from } 0 \cdots 01 X_{t_0} \cdots X_1 ) \\
~~~~~~H^2:  0 \cdots 01\underbrace{0 \cdots 01\underbrace{0 \cdots 0\underbrace{X_{t_2} \cdots X_{1}}_{t_2}}_{t_1}}_{t_0}
\\\text { (take a } t_2 \text {-dimensional sub-$H_n^4$ graph }\\
\text{from } 0 \cdots 010 \cdots 01 X_{t_1} \cdots X_1 ) \\
~~~~~~\cdots
$.

Let $G_{k,j}$  be denoted by $H_n^4\left[V(H^k)\cup V(H^j)\right]$, and let $H_n^4\left[L_m^n\right]$ be denoted by $ H_n^4\left[V\left(H^0\right) \cup \cdots \cup V\left(H^s\right)\right]$. Pay attention to the fact that
 each $H_n^4[V(H^i)]$ is a $t_i$-dimensional sub-$H_n^4$ graph,   for $0\leq i\leq s$, except for  $t_s=1$, $H^s$ is the complete graph $K_2$ and if $t_s=0$, $H^s$ has exactly one vertex.
 In light of the definition of $H_n^4$ and the structure of $H_n^4\left[L_m^n\right]$, there exists a matching of size $2^{t_k}$ between $V(H^{k})$ and $V(H^{j})$ in the subgraph $G_{k,j}$ for $0\leq j<k\leq s$ and $t_j\geq 2$ but for $t_j=t_{s-1}=1$ and $t_k=t_s=0$,  there exist two edges between $V(H^{s-1})$ and $V(H^{s})$ in the subgraph $G_{s,s-1}$.

 The size of edges of $H_n^4\left[L_m^n\right]$  can  be figured out
 by considering the edges within $H^{i,}s\left( \left( \sum_{i=0}^s t_i 2^{t_i}+4p\right)/2\right)$ and the edges between $H^{i,}s\left( \sum_{i=0}^s 2 \cdot i \cdot 2^{t_i}/2\right.$ for $0\leq q\leq2$ or $ \left(\sum_{i=0}^s 2 \cdot i \cdot 2^{t_i}\right)/2+1$ for $\left.q=3\right)$.
 Thus, the claim holds, i.e., $e x_m(H_n^4) \geq f(m)=2|E(H_n^4[L_m^n])|$ and consequently $e x_m(H_n^4)=f(m)=2|E(H_n^4[L_m^n])|$.
\end{proof}
\begin{figure}[tb]
	\begin{center}
		 \includegraphics[width=1\linewidth]{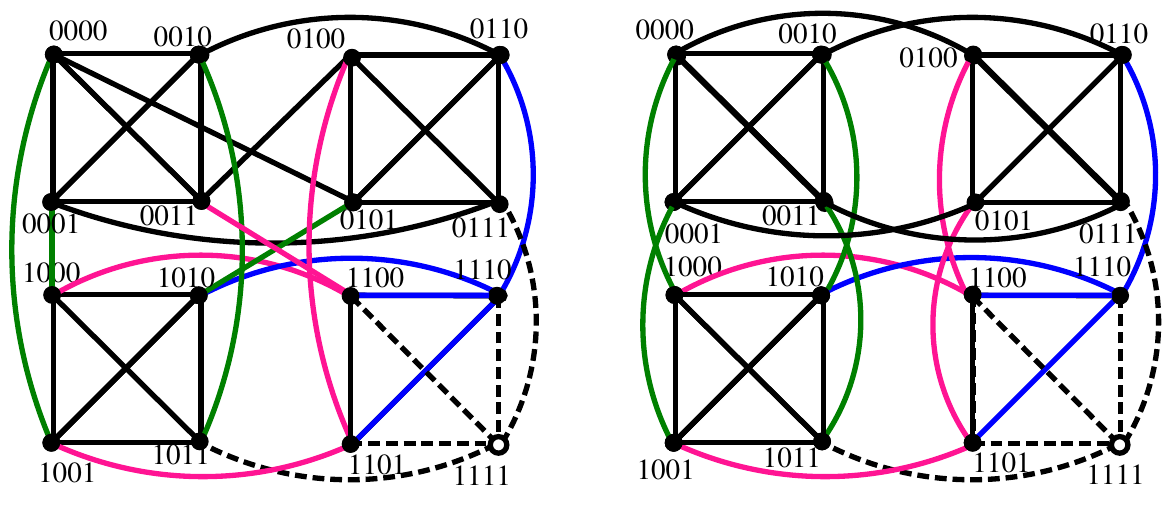}
		\caption{Induced subgraphs   $H_4^4[L_{15}^{~4}]$ and $Q_{4,3}[L_{15}^{~4}]$.}
\label{Induced subgraphs}
	\end{center}
\end{figure}
As an example, assume that $n=4$, $m=15$, and $S_{15}=\{0$,\,$1$,\,$2$,\,$3$,\,$4$,\,$5$,\,$6$,\,$7$,\,$8$,\,$9$,\,$10$,\,$11$,\,$12$,\,$13$,\,$14\}$.
Then $L_{15}^4=\{0000$,\,$0001$,\,$0010$,\,$0011$,\,$0100$,\,$0101$,\,$0110$,\,$0111$,\,$1000$,\,$1001$,\,
$1010$,\,$1011$,\,$1100$,\,$1101$,\,$1110\}$.
Since $15=2^3+2^2+2^1+2^0$ and $15=4\times3+3$,
 we can calculate that $ex_{15}(H_n^4)=ex_{15}(Q_{n,n-1})=2|E(H_n^4[L_{15}^n])|=3\times2^3+2\times2^2+1\times2^1+0\times2^0+2\times0\times2^3+ 2\times1\times2^2+2\times2\times2^1+2\times3\times2^0+4\times3+2 =70$ by  Theorem \ref{thmex_m(H_n^4)=f(m)}.
The subgraphs induced by $L_{15}^4$ in a $4$-dimensional $K_4$-hypercube and an enhanced hypercube $Q_{4,3}$ are demonstrated in Fig.~\ref{Induced subgraphs}.
The edges within sub-$H_4^4$ graphs are colored in solid black lines. The edges between $H_3^4$ and $H_2^4$, $H_2^4$ and  $K_2$, $K_2$ and $K_1$ are colored in
green, pink, and blue, respectively.
They form the induced subgraph $H_4^4[L_{15}^4]$.
\section{Some properties  of the optimal solution of the edge isoperimetric problem $\xi_m(H_n^4)$}
\label{sec4}
In this section, we present several lemmas that describe the minimum edge isoperimetric number $\xi_m(H_n^4)$, which offers an upper bound of $\lambda(H_n^4)$. These lemmas shed light on the significance of the monotonic interval and fractal structure in determining the specific value of the function $\lambda_h\left(H_n^4\right)$.
\begin{lem}
For $n\geq2$, $1<m<2^{n}$, then $\xi_m(H_n^4)>0$.
\end{lem}
\begin{proof}
For $n\geq2$, $1<m<2^{n}$, the edge-cut$[L_m^n,\overline{L_m^n}]$ of the connected graph $H_n^4$ is not empty. Thus, $\xi_m(H_n^4)=\left| [L_m^n,\overline{L_m^n} ]\right|>0$. The proof is completed.
\end{proof}
\begin{lem}\label{lem4.4}
Let  $n,\,c,\,m$ be integers with $n\geq4,~0\leq c\leq n-2$, and $2^c\leq m\leq2^{n-1}$. Then $\xi_{m}(H_n^4)\geq \xi_{2^c}(H_n^4).$
\end{lem}
\begin{proof}
We  transform the claim  by considering $\xi_{m}(H_n^4)\geq \xi_{2^t}(H_n^4)$, with $2^t<m \leq 2^{t+1}$, $c \leq t \leq n-2$. Let $m=\sum_{i=0}^s 2^{t_i}$ be the decomposition of $m$.

As  $2^t<m \leq 2^{t+1}$, it follows that $0<m^{\prime}=m-2^t \leq 2^t$. We need only consider three cases.
If $t=0$, then $m=2$ and $\xi_2(H_n^4)-\xi_1(H_n^4)=n-1 \geq 0$. If $t=1$, then $2<m \leq 4$. It follows that $\xi_3(H_n^4)-\xi_2(H_n^4)=n-3 \geq 0$  and $\xi_4(H_n^4)-\xi_2(H_n^4)=2n-8 \geq 0$ .

Now, turn to the case $2 \leq t \leq n-2$. Let $k$ be an integer satisfying $t_k \geq 2$, $t_{k+1} \leq 1$, and let $m=4 p+q$ with $0 \leq q<4$. Then $4 p=\sum_{i=0}^k 2^{t_i}$ and $t_0=t$. By Lemma \ref{lem2.3} (3), we have
\par If $0 \leq q\leq2$, then
\\$
~~\xi_{m}(H_n^4)-\xi_{2^t}(H_n^4)\\
=\left[(n+1) m-e x_m\left(H_n^4\right)\right]-\left[(n+1) 2^t-e x_{2 ^t}\left(H_n^4\right)\right] \\
=(n+1) m^{\prime}-\left(e x_m\left(Q_n\right)+4 p\right)+\left(e x_{2^t}\left(Q_n\right)+2^{t_0}\right) \\
=(n+1) m^{\prime}-\left(e x_{m^{\prime}}\left(Q_n\right)+2 m^{\prime}\right)-\sum_{i=1}^k 2^{t_i} \\
=(n-1) m^{\prime}-ex_{m^{\prime}}\left(Q_n\right)-\sum_{i=1}^k 2^{t_ i} \\
\geq(n-1) m^{\prime}-t m^{\prime}-\sum_{i=1}^k 2^{t_i} \\
=(n-1-t) m^{\prime}-\sum_{i=1}^k 2^{t_i} \\
\geq m^{\prime}-\sum_{i=1}^k 2^{t_i}=q \geq 0.$

If $q=3$, then $t_{k+1}=1,~t_{k+2}=0$. Thus,
\\$
~~\xi_{m}(H_n^4)-\xi_{2^t}(H_n^4)\\
=\left[(n+1) m-e x_m(H_n^4)\right]-\left[(n+1) 2^t-e x_{2^ t}\left(H_n^4\right)\right] \\
=(n+1) m^{\prime}-\left(e x_m\left(Q_n\right)+4 p+2\right)+\left(e x_{2^t}\left(Q_n\right)+2^{t_0}\right) \\
=(n+1) m^{\prime}-\left(e x_{m^{\prime}}\left(Q_n\right)+2 m^{\prime}\right)-\sum_{i=1}^k 2^{t_i}-2 \\
\geq(n-1) m^{\prime}-t{m^{\prime}}-\sum_{i=1}^k 2^{t_i}-2 \\
=(n-1-t) m^{\prime}-\sum_{i=1}^k 2^{t_i}-2 \\
\geq m^{\prime}-\sum_{i=1}^k 2^{t_i}-2=q-2 \geq 0.$
\par In summary, the proof is completed.
\end{proof}
Lemma \ref{lem4.2} illustrates the monotonic increasing property of $\xi_m(H_n^4)$ for $1\leq m \leq2^{\lceil n/2  \rceil }-1$ and paves the way for Theorem \ref{thm1}.
\begin{lem}\label{lem4.2}
For $n\geq3$, we have \\
(1). the function $\xi_m(H_n^4)$ is monotonically increasing for the interval $1\leq m \leq2^{\lceil n/2  \rceil }-1$. In other words, $\xi_{m+1}(H_n^4)-\xi_m(H_n^4)\geq0$ for $n\geq3,1\leq m \leq2^{\lceil n/2  \rceil }-2$; and\\
(2). $\xi_m(H_n^4) > \xi_{2^{\lceil n / 2\rceil}-2-\gamma}(H_n^4) =\xi_{2^{\lceil n / 2\rceil}}(H_n^4)$ for $2^{\lceil n / 2\rceil}-2-\gamma<m<2^{\lceil n / 2\rceil}$.
\end{lem}
\begin{proof}
(1). By Lemma \ref{lem2.3} (2), for each integer $m=4 p+q$, $0 \leq q<4$, we have $ex_m\left(H_n^4\right)-e x_{m+1}\left(H_n^4\right)=ex_m\left(Q_n\right)-e x_{m+1}\left(Q_n\right)=-2 s-2$ for $q=0,1$ and $ex _m\left(H_n^4\right)-e x_{m+1}\left(H_n^4\right)=e x_m\left(Q_n\right)-e x_{m+1}\left(Q_n\right)-2=-2s-4$ for $q=2,3$.
Let $m=\sum_{i=0}^s 2^{t_i}$ be the decomposition of $m$. For $1 \leq m \leq 2^{\lceil n/2  \rceil }-2$, we have $s \leq\lceil n/2  \rceil -2$.

If $q=0,1$, then
\\$
~~\xi_{m+1}(H_n^4)-\xi_m(H_n^4) \\
= \!{\left[(n+1)(m+1)\!-\!ex_{m+1}\left(H_n^4\right)\right]-\!\left[(n+1) m-\!e x_m\left(H_n^4\right)\right] } \\
=(n+1)+\left[e x_m\left(H_n^4\right)-e x_{m+1}\left(H_n^4\right)\right] \\
=(n+1)-2 s-2
= n-2 s-1
 \geq 0.$

If $q=2,3$, then
\\$
~~\xi_{m+1}(H_n^4)-\xi_m(H_n^4) \\
=\!{\left[(n+1)(m+1)-\!e x_{m+1}\left(H_n^4\right)\right]-\!\left[(n+1)m-\!e x_m\left(H_n^4\right)\right] } \\
=n+1-2 s-4
=  n-2 s-3 \geq 0.
$
\par In summary, $\xi_{m+1}(H_n^4)-\xi_m(H_n^4)\geq0$.\\
(2).  First of all, we can  calculate
 $\xi_{2^{\lceil n / 2\rceil}}(H_n^4)
 =(n+1)2^{\lceil n / 2\rceil}-ex_{2^{\lceil n / 2\rceil}}(H_n^4)
 =(n+1)2^{\lceil n / 2\rceil}-ex_{2^{\lceil n / 2\rceil}}(Q_n)-2^{\lceil n / 2\rceil}
 =(n+1)2^{\lceil n / 2\rceil}-\lceil n / 2\rceil2^{\lceil n / 2\rceil}-2^{\lceil n / 2\rceil}=\lfloor n/2\rfloor2^{\lceil n / 2\rceil}$.
 Moreover, it is not difficult to find that $ex_{2^{\lceil n / 2\rceil}-2}(Q_n)=ex_{2^{n/2}}(Q_n)-2[3+2(n/2-2)]=(n/2)2^{n/2}-2n+2$ for $n$ is even and $ex_{2^{\lceil n / 2\rceil}-3}(Q_n)=ex_{2^{\lceil n/2\rceil}}(Q_n)-2[4+3(\lceil n / 2\rceil-2)]=\left\lceil n / 2\right\rceil2^{\lceil n / 2\rceil}-3n+1$ for $n$ is odd. By Lemma~\ref{lem4.2}, $~\xi_{m+1}(H_n^4)-\xi_m(H_n^4)\geq 0$ for $1 \leq m \leq 2^{\lceil n / 2\rceil}-2$.

  It suffices to show that $\xi_{2^{\lceil n / 2\rceil}-2-\gamma}(H_n^4)=\xi_{2^{\lceil n/2\rceil}}(H_n^4)$.
For  odd $n$, by calculating, we have $\xi_{2^{\lceil n/2\rceil}-3}(H_n^4)=(n+1)\left(2^{\lceil n/2\rceil}-3\right)-e x_{2^{\lceil n/2\rceil}-3}\left(H_n^4\right)=(n+1)\left(2^{\lceil n/2\rceil}-3\right)-e x_{2^{\lceil n/2\rceil}-3}\left(Q_n\right)+2^{\lceil n/2\rceil}-3-1)
=\lfloor n/2\rfloor 2^{\lceil n/2\rceil}=\xi_{2^{\lceil n/2\rceil}}(H_n^4)$.
For   even $n$, $\xi_{2^{\lceil n/2\rceil-2}}(H_n^4)=\xi_{2^{\lceil n/2\rceil}}(H_n^4)$ is proved analogously.
The proof holds.
\end{proof}
Let $n$ be an integer with $n\geq3$. Assume that $t=0,1, \cdots,\lfloor n / 2\rfloor -1$, $g_t=\lceil\left(2^{2 t+2+\gamma} \right) / 3\rceil$.
To address the range $2^{\lceil n / 2\rceil+t}-g_t \leq m \leq 2^{\lceil n / 2\rceil+t}$  effectively, we refine the intervals by introducing a new subdivision.
For $t=0,1, \cdots, \lfloor n/2\rfloor-1$ and $0 \leq d_t \leq t+1$, the definition of function   $m_{t, d_t}$ is\\
\begin{math}
\begin{aligned}
m_{t, d_t}=\begin{cases}
2^{\lceil n / 2\rceil+t} &\text { if }d_t=0 ; \\
2^{\lceil n / 2\rceil+t}-\sum_{i=0}^{d_t-1} 2^{2 t-2 i+\gamma} &\text{ if }1 \leq d_t \leq t;\\
2^{\lceil n / 2\rceil+t}-\sum_{i=0}^{d_t-1} 2^{2 t-2 i+\gamma}-1 &\text { if } d_t=t+1.\end{cases}
\end{aligned}
\end{math}\\
\begin{figure}[!t]
\flushleft
\includegraphics[width=1\linewidth]{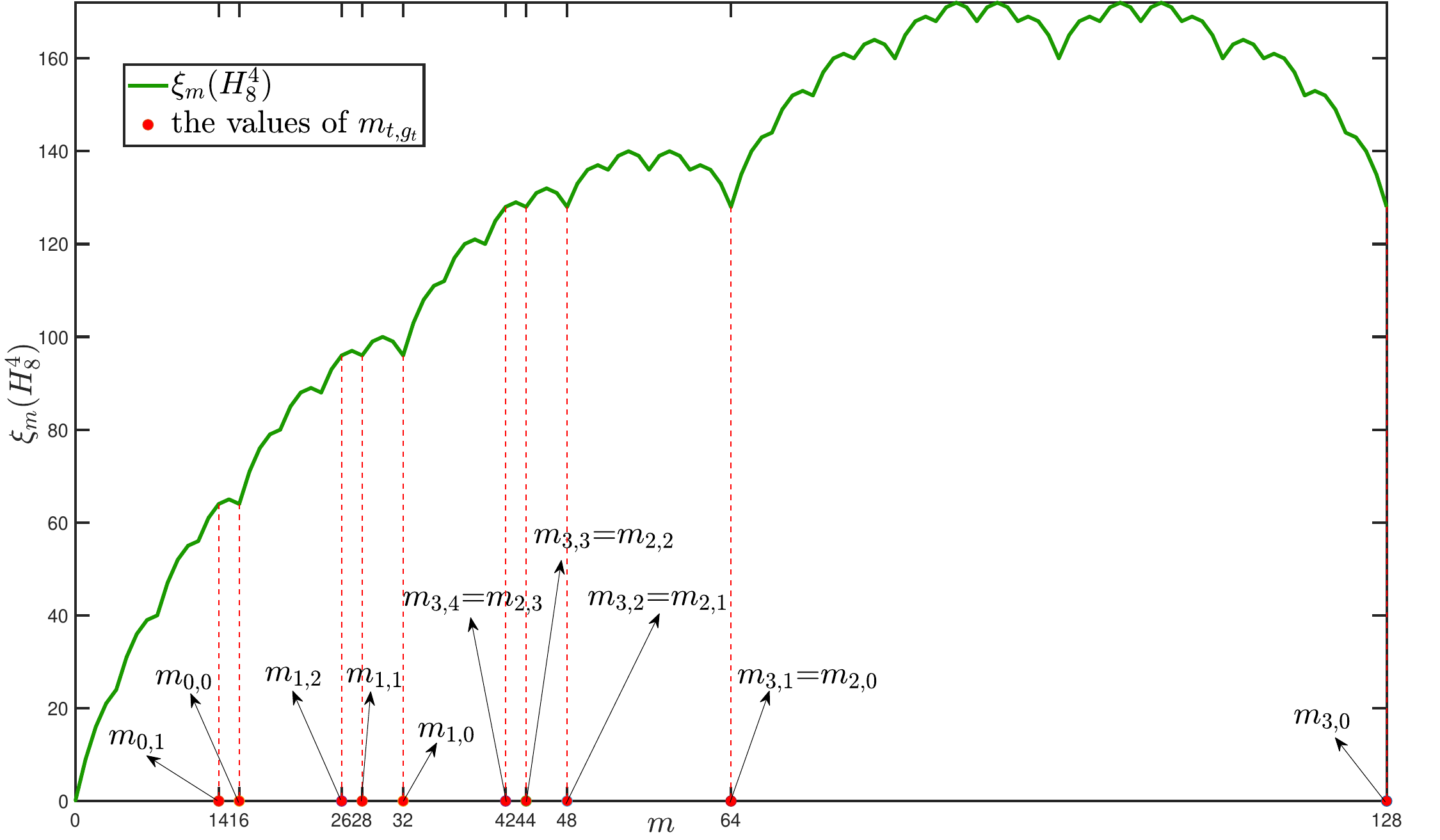}
		\caption{The values of $m_{t,g_t}$ for $n=8$.}
\label{figThevaluesofM_{r,k_r}}
\end{figure}
For a better understanding,
in Fig. \ref{figThevaluesofM_{r,k_r}}, all the values of $m_{t, d_t}$ for $n=8$ are marked on the $x$-axis with red dots, depending on the image of the function $\xi_m(H_8^4)$.

\begin{lem}\label{lem4.5}
Let $n$ be an integer with $n\geq3$. Suppose that $2^{\lceil n / 2\rceil+t}-g_t \leq m \leq 2^{\lceil n / 2\rceil+t}$, for $t=0,1, \cdots,\lfloor n / 2\rfloor -1$, $g_t=\lceil\left(2^{2 t+2+\gamma}\right) / 3\rceil$.
Then $\xi_m\left(H_n^4\right) \geq \xi_{2^{\lceil n / 2\rceil+t}}\left(H_n^4\right)$.
\end{lem}
\begin{proof}
It has already been proven in Lemma~\ref{lem4.2} for $t=0$. The discussion below pertains to the case for $t \geq 1$. One can check that $m_{t, d_t}-m_{t, d_t+1}=2^{2 t-2 d_t+\gamma}$ for $0 \leq d_t \leq t-1,~m_{t,t}-m_{t, t+1}=2+\gamma$ for $d_t=t$ and $m_{t, t+1}=2^{\lceil n / 2\rceil+t}-g_t$.

By calculation, for $q=0$, $1$ and $2$, it can be obtained that
\\$
~~\xi_m(H_n^4)\\
=(n+1) m-e x_m\left(H_n^4\right)\\
=(\lfloor n/2\rfloor-t+1)m+\left(\lceil n/2\rceil+t\right)m-ex_m(H_n^4)\\
=\!(\lfloor n/2\rfloor-t+1)m+(\lceil n/2\rceil+t)m-\!ex_m(Q_{\lceil n/2\rceil+t})\!\!-m+q\\
=(\lfloor n/2\rfloor-t)m+(\lceil n/2\rceil+t)(2^{\lceil n/2\rceil+t}-m)\\
~~-ex_{2^{\lceil n/2\rceil+r}-m}(Q_{\lceil n/2\rceil+t})+q\\
=(\lfloor n/2\rfloor-t)2^{\lceil n/2\rceil+t}-(\lfloor n/2\rfloor-t)(2^{\lceil n/2\rceil+t}-m)\\
~~+(\lceil n/2\rceil+t)(2^{\lceil n/2\rceil+t}-m)
-ex_{2^{\lceil n/2\rceil+t}-m}(Q_{\lceil n/2\rceil+t})
+q\\
=\xi_{m_{t,0}}(H_n^4)+(2t+f)(2^{\lceil n/2\rceil+t}-m)\\
~~-ex_{2^{\lceil n/2\rceil+t}-m}(Q_{\lceil n/2\rceil+t})+q.
$

For $q=3$, we can determine that\\
$~~\xi_m(H_n^4)
%=(n+1) m-e x_m\left(H_n^4\right)
=(n+1)m-ex_m(Q_n)-m+1
=\xi_{m_{t,0}}(H_n^4)+(2t+\gamma)(2^{\lceil n/2\rceil+t}-m)
-ex_{2^{\lceil n/2\rceil+t}-m}(Q_{\lceil n/2\rceil+t})+1.
$

Next, we  calculate $ex_{2^{\lceil n/2\rceil+t}-m}(Q_{\lceil n/2\rceil+t})$.
\par\textbf{Case 1.} $m_{t,t}\leq m\leq m_{t,0}$.
\par In this case, assume $u=m_{t,d_t}-m$.\\
$~~ex_{2^{\lceil n/2\rceil+t}-m}(Q_{\lceil n/2\rceil+t})\\
=ex_{2^{\lceil n/2\rceil+t}-m_{t,d_t}}(Q_{\lceil n/2\rceil+t})+ex_{u}(Q_{\lceil n/2\rceil}+t)+2d_tu\\
=\sum_{i=0}^{d_t-1}(2t-2i+\gamma)2^{2t-2i+\gamma}+\sum_{i=0}^{d_t-1}2i2^{2t-2i+\gamma}\\
~~+ex_{u}(Q_{\lceil n/2\rceil}+t)+2d_tu\\
=(2t+\gamma)\sum_{i=0}^{d_t-1}2^{2t-2i+\gamma}+ex_u(Q_{\lceil n/2\rceil}+t)+2d_tu\\
=(2t+\gamma)(2^{\lceil n/2\rceil+t}-m_{t,d_t})+ex_u(Q_{\lceil n/2\rceil}+t)+2d_tu.
$

 \par\textbf{Subcase 1.1.} $q=0,\ 1$ and $2$.
 \par Combining the above formulas, we can obtain that\\
 $~~\xi_m(H_n^4)\\
 =(n+1)m-ex_m(H_n^4)\\
=\xi_{m_{t,0}}(H_n^4)+(2t+\gamma)(m_{t,d_t}-m)-ex_{u}(Q_{\lceil n/2\rceil+t})\\
~~-2d_tu+q\\
=\xi_{m_{t,0}}(H_n^4)+(2t-2d_t+\gamma)u-ex_{u}(Q_{2t-2d_t+\gamma})+q\\
=\xi_{m_{t,0}}(H_n^4)+\xi_u(Q_{2t-2d_t+\gamma})+q\\
\geq \xi_{m_{t,0}}(H_n^4).
 $
 \par For the last inequality, the equality is achieved if and only if  $u=0$ or $u=2^{2t-2d_t+\gamma}$.
\par\textbf{Subcase 1.2.} $q=3$.
 \par Similar to Subcase 1.1, we have\\
 $~~\xi_m(H_n^4)\\
 %=(n+1)m-ex_m(H_n^4)\\
=\xi_{m_{t,0}}(H_n^4)+(2t+\gamma)(m_{t,d_t}-m)-ex_{u}(Q_{\lceil n/2\rceil+t})\\
~~-2d_tu+1\\
%=\xi_{m_{t,0}}(H_n^4)+(2t-2d_t+\gamma)u-ex_{u}(Q_{2t-2d_t+\gamma})+1\\
%=\xi_{m_{t,0}}(H_n^4)+\xi_u(Q_{2t-2d_t+\gamma})+1\\
\geq \xi_{m_{t,0}}(H_n^4).
 $
 \par For the last inequality, the equality is achieved if and only if  $u=0$ or $u=2^{2t-2d_t+\gamma}$.

\par\textbf{Case 2.} $m_{t,t+1}\leq m\leq m_{t,t}$.
\par In this case, assume $v=m_{t,t}-m$.\\
$~~ex_{2^{\lceil n/2\rceil+t}-m}(Q_{\lceil n/2\rceil+t})\\
=ex_{2^{\lceil n/2\rceil+t}-m_{t,t}}(Q_{\lceil n/2\rceil+t})+ex_{v}(Q_{\lceil n/2\rceil+t})+2tv\\
=\sum_{i=0}^{t-1}(2t-2i+f)2^{2t-2i+f}+\sum_{i=0}^{t-1}2i2^{2t-2i+f}\\
~~+ex_{v}(Q_{\lceil n/2\rceil+t})+2tv\\
=(2t+\gamma)\sum_{i=0}^{t-1}2^{2t-2i+\gamma}+ex_v(Q_{\lceil n/2\rceil}+t)+2tv\\
=(2t+\gamma)(2^{\lceil n/2\rceil+t}-m_{t,t})+ex_v(Q_{\lceil n/2\rceil}+t)+2tv.
$
 \par\textbf{Subcase 2.1.} $q=0,\ 1$ and $2$.
  \par By the above formulas, \\
 $~~\xi_m(H_n^4)\\
=(n+1)m-ex_m(H_n^4)\\
=\xi_{m_{t,0}}(H_n^4)+(2t+\gamma)(m_{t,t}-m)-ex_{v}(Q_{\lceil n/2\rceil+t})-2tv+q\\
=\xi_{m_{t,0}}(H_n^4)+(2t-2t+\gamma)v-ex_{v}(Q_{2+\gamma})+q\\
=\xi_{m_{t,0}}(H_n^4)+\xi_v(Q_{2+\gamma})-2v+q\\
\geq \xi_{m_{t,0}}(H_n^4).
 $
 \par For the last inequality, the equality holds if and only if  $v=0$ or $v=2+\gamma$.
\par\textbf{Subcase 2.2.} $q=3$.
 \par By using the above formulas, then\\
 $~~\xi_m(H_n^4)\\
%=(n+1)m-ex_m(H_n^4)\\
=\xi_{m_{t,0}}(H_n^4)+(2t+\gamma)(m_{t,t}-m)-ex_{v}(Q_{\lceil n/2\rceil+t})-2tv+1\\
%=\xi_{m_{t,0}}(H_n^4)+(2t-2t+\gamma)u-ex_{v}(Q_{2+\gamma})+1\\
=\xi_{m_{t,0}}(H_n^4)+\xi_v(Q_{2+\gamma})-2v+1\\
\geq \xi_{m_{t,0}}(H_n^4).
 $
 \par For the last inequality, the equality is achieved if and only if  $v=0$ or $v=2+\gamma$.

 In summary, $\xi_{m_{t,d_t}}(H_n^4)=\xi_{m_{t,0}}(H_n^4)=\xi_{m_{t,d_t+1}}(H_n^4)$. The proof is completed.
\end{proof}

\section{Application to five link fault patterns}\label{sec5}
We  prove the main theorems by applying the optimal solution of the edge isoperimetric problem to evaluate the reliability
of the generalized $K_4$-hypercubes,  and subsequently present the data vividly with the aid of tables and  graphs.
\subsection{The $h$-extra edge-connectivity of $\mathscr{H}_n^4$}
In this subsection,  the function $\lambda_h(H_n^4)$ through a monotonically increasing interval and  main intervals of concentration phenomena are explored.
\begin{figure*}[!t]
	\flushleft
\includegraphics[width=1\linewidth]{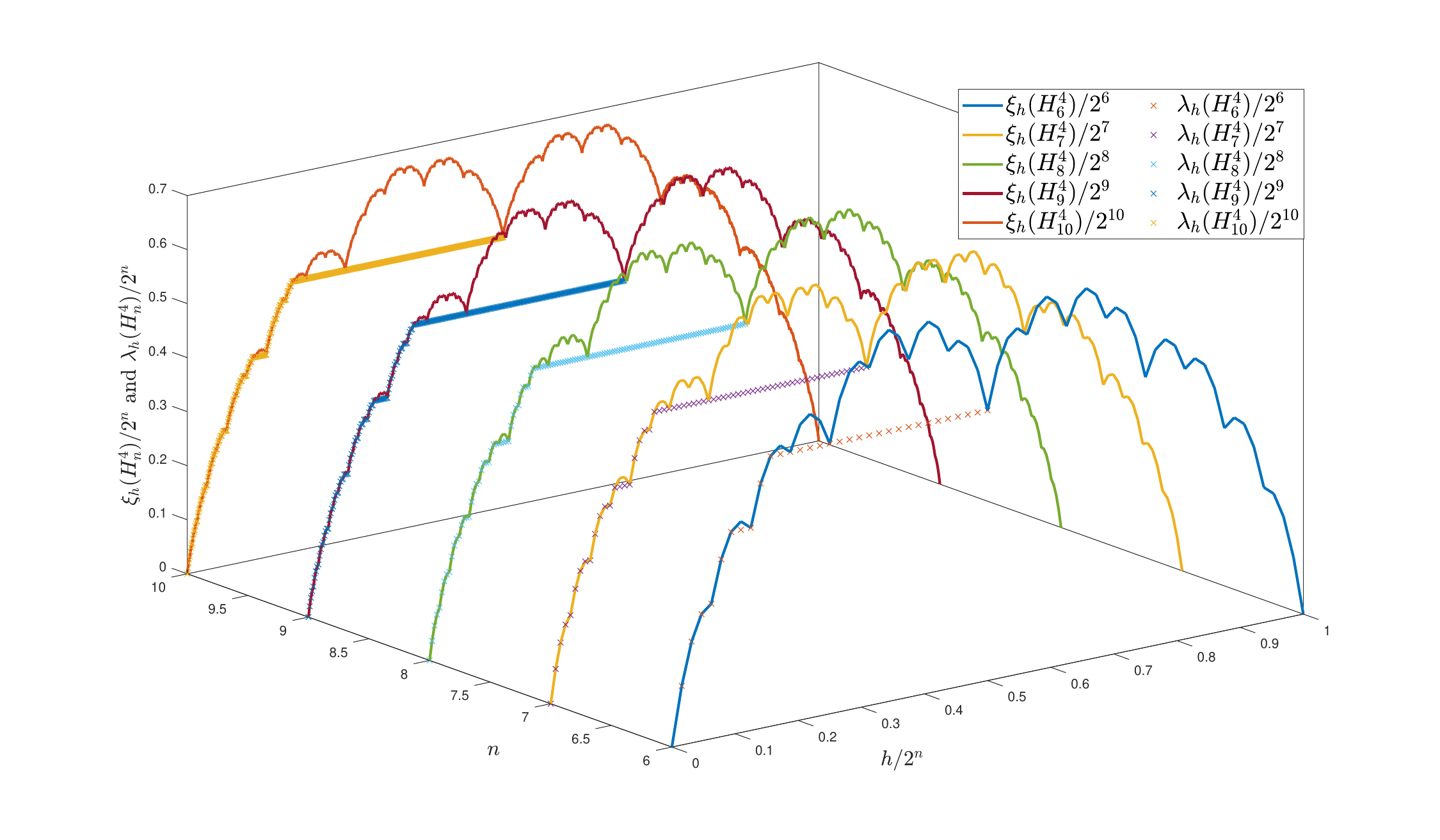}
		\caption{The normalized values of $\xi_h(H_n^4) $ and $\lambda_h(H_n^4) $ for $6\leq n\leq10$.}
\label{fignormalized values}
\end{figure*}

\subsubsection{A monotonically increasing range of function $\lambda_h(H_n^4)$}
\noindent\textbf{The proof of Theorem~\ref{thm1}}\quad
For $n=3$,  $2^{\lceil n/2\rceil}-2-\gamma=1$. Thus, $\lambda_1(H_n^4)=\min\left\{\xi_m(H_n^4): h \leq m \leq 1\right\}=\xi_1(H_n^4)$.

For $n\geq4$, $1\leq h\leq 2^{\lceil n/2\rceil}-2-\gamma$, by formula (\ref{equa2}), Lemmas \ref{lem4.2}, and \ref{lem4.4}, it suffices to have
\\$
~~\lambda_h(H_n^4)\\
=\min\Big\{\xi_m(H_n^4):h\leq m\leq2^{n-1}\Big\}\\
=\min\Big\{\xi_m(H_n^4):h\leq m\leq2^{\lceil n/2\rceil}\Big\}\\
=\min\Big\{\xi_m(H_n^4): h \leq m \leq 2^{\lceil n/2\rceil}-2-\gamma\Big\}\\
=\xi_h(H_n^4)$.
\par The proof is finished.
$\hfill\blacksquare$
\subsubsection{Main intervals of concentration phenomena of $\lambda_h(H_n^4)$}
\textbf{The proof of Theorem~\ref{thm2}}\quad
For $n\geq3$ and $2^{\lceil n/2\rceil+t}-g_t\leq h\leq2^{\lceil n/2\rceil+t}$, by using the Formula (\ref{equa2}), Lemmas \ref{lem4.4} and \ref{lem4.5}, we have\\$
~~\lambda_h(H_n^4)\\
=\min\Big\{\xi_m(H_n^4):h\leq m\leq2^{n-1}\Big\}\\
=\min\Big\{\xi_m(H_n^4):h\leq m\leq2^{\lceil n/2\rceil+t}\Big\}\\
=\xi_{2^{\lceil n/2\rceil+t}}(H_n^4)\\
=(\lfloor n/2\rfloor-t)2^{\lceil n/2\rceil+t}$.
\par The proof is finished.
$\hfill\blacksquare$
\begin{cor}\label{cor}
For $t=\lfloor n/2\rfloor-1$, we obtain that $\lambda_h(H_n^4)=2^{n-1}$, for $\lfloor 2^{n-1}/3\rfloor\leq h\leq 2^{n-1}$.
\end{cor}
\begin{proof}
In fact,
through calculation, if $t=\lfloor n/2\rfloor-2$ and $n$ is even, we have
 $m_{\lfloor n/2\rfloor-2,\lfloor n/2\rfloor-1}=2^{\lceil n/2\rceil+\lfloor n/2\rfloor-2}-\sum_{i=0}^{\lfloor n/2\rfloor-2}2^{n-4-2i}-1=2^{n-2}-2^{n-4}-2^{n-6}-\cdots-2^{2}-2^{0}-1=2^{n-2}-2^{n-4}-2^{n-6}-\cdots-2^{2}-2^{1}$ and $2m_{\lfloor n/2\rfloor-2,\lfloor n/2\rfloor-1}=2^{n-1}-2^{n-3}-2^{n-5}-\cdots-2^{3}-2^{2}-1$. Then
 $3m_{\lfloor n/2\rfloor-2,\lfloor n/2\rfloor-1}=2^{n-1}+2^{n-2}-2^{n-3}-2^{n-4}-\cdots-2^{3}-2^{2}-2^{2}-2=2^{n-1}-2$. So
 $m_{\lfloor n/2\rfloor-2,\lfloor n/2\rfloor-1}=(2^{n-1}-2)/3=\lfloor 2^{n-1}/3\rfloor$.
Similarly, if $t=\lfloor n/2\rfloor-2$ and $n$ is odd, we have
% $M_{\lfloor n/2\rfloor-2,\lfloor n/2\rfloor-1}
 %=2^{\lceil n/2\rceil+\lfloor n/2\rfloor-2}-\sum_{i=0}^{\lfloor n/2\rfloor-2}2^{n-4-2i}-1=2^{n-2}-2^{n-4}-2^{n-6}-\cdots-2^{3}-2^{1}-1$ and
 %$2M_{\lfloor n/2\rfloor-2,\lfloor n/2\rfloor-1}=2^{n-1}-2^{n-3}-2^{n-5}-\cdots-2^{4}-2^{2}-2$. Then
% $3M_{\lfloor n/2\rfloor-2,\lfloor n/2\rfloor-1}=2^{n-1}+2^{n-2}-2^{n-3}-2^{n-4}-\cdots-2^{2}-2^{1}-2^{1}-1=2^{n-1}-1$. So
 $m_{\lfloor n/2\rfloor-2,\lfloor n/2\rfloor-1}=(2^{n-1}-1)/3=\lfloor 2^{n-1}/3\rfloor$.
In consequence, $m_{\lfloor n/2\rfloor-2,\lfloor n/2\rfloor-1}=\lfloor 2^{n-1}/3\rfloor$.
Furthermore, it can be  acquired that $m_{\lfloor n/2\rfloor-2,\lfloor n/2\rfloor-1}=m_{\lfloor n/2\rfloor-1,\lfloor n/2\rfloor}$. Based on Theorem \ref{thm2}, it can be derived that $\lambda_h(H_n^4)=2^{n-1}$ for $\lfloor 2^{n-1}/3\rfloor\leq h\leq 2^{n-1}$.
\end{proof}
By Corollary \ref{cor},
$m_{\lfloor n/2\rfloor-2,\lfloor n/2\rfloor-1}$$=m_{\lfloor n/2\rfloor-1,\lfloor n/2\rfloor}=\lfloor 2^{n-1}/3\rfloor$. Thus, it can be seen that $\left[\lfloor 2^{n-1}/3\rfloor,2^{n-2}\right]\subseteq\left[\lfloor 2^{n-1}/3\rfloor,2^{n-1}\right]$.
Therefore, we  disregard the scenario for $t=\lfloor n/2\rfloor-2$.
\begin{rem}
The lower and upper bounds of $h$ in $\lambda_h(H_n^4)$ are both sharp for $2^{\lceil n/2\rceil+t}-g_t\leq h\leq2^{\lceil n/2\rceil+t}$, $t=0,1,\cdots,\lfloor n/2\rfloor-3,\lfloor n/2\rfloor-1$.
\end{rem}
(1)
On the one hand, if $n$ is even, according to the definition of the function $m_{t,d_t}$, it follows that $m_{t,t+1}=2^{\lceil n/2\rceil+t}-\sum_{i=0}^{t}2^{2t-2i}-1$. Note that $2^{\lceil n/2\rceil+t}=2^{\lceil n/2\rceil+t-1}+2^{\lceil n/2\rceil+t-2}+\cdots+2^2+2^1+2^0+2^0$ contains $(\lceil n/2\rceil+t+1)$ terms and $\sum_{i=0}^{t}2^{2t-2i}=2^{2t}+2^{2t-2}+\cdots+2^2+2^0$ contains $(t+1)$ terms. Then $m_{t,t+1}$ contains $(\lceil n/2\rceil-1)$ terms and the remainder of $m_{t,t+1}$ divided by $4$ is $2$. It can be observed that $m_{t,t+1}-1$ contains $(\lceil n/2\rceil-1)$ terms and the remainder of $m_{t,t+1}-1$ divided by $4$ is $1$.
Then,\\
$~~\xi_{m_{t,t+1}}(H_n^4)-\xi_{m_{t,t+1}-1}(H_n^4)\\
=(n+1)m_{t,t+1}-ex_{m_{t,t+1}}(H_n^4)-\left[(n+1)(m_{t,t+1}-1)\right.\\
~~\left.-ex_{m_{t,t+1}-1}(H_n^4)\right]\\
=n+1-\left[ex_{m_{t,t+1}}(H_n^4)-ex_{m_{t,t+1}-1}(H_n^4)\right]\\
=n+1-\left[ex_{m_{t,t+1}}(Q_n)+{m_{t,t+1}}-2-ex_{m_{t,t+1}-1}(Q_n)\right.\\
~~\left.+{m_{t,t+1}-1}-1\right]\\
=n+1-\left[ex_{m_{t,t+1}}(Q_n)-ex_{m_{t,t+1}-1}(Q_n)\right]\\
=n+1-\left[2(\lceil n/2\rceil-2)+2\right]\\
=n+1-(n-2)
=3>0.
$

On the other hand, it can be calculated that $m_{t,t+1}=2^{\lceil n/2\rceil+t}-\sum_{i=0}^{t}2^{2t-2i+1}-1$ for odd $n$. Note that $\sum_{i=0}^{t}2^{2t-2i+1}=2^{2t+1}+2^{2t-1}+\cdots+2^3+2^1$ contains $(t+1)$ terms. Then $m_{t,t+1}$ contains $(\lceil n/2\rceil-1)$ terms and the remainder of $m_{t,t+1}$ divided by $4$ is $1$. It can be observed that $m_{t,t+1}-1$ contains $(\lceil n/2\rceil-2)$ terms and is divisible by $4$.
Then,\\
$~~\xi_{m_{t,t+1}}(H_n^4)-\xi_{m_{t,t+1}-1}(H_n^4)\\
=(n+1)m_{t,t+1}-ex_{m_{t,t+1}}(H_n^4)-\left[(n+1)(m_{t,t+1}-1)\right.\\
\left.-ex_{m_{t,t+1}-1}(H_n^4)\right]\\
=n+1-\left[ex_{m_{t,t+1}}(H_n^4)-ex_{m_{t,t+1}-1}(H_n^4)\right]\\
=n+1-\left[ex_{m_{t,t+1}}(Q_n)+{m_{t,t+1}}-1-ex_{m_{t,t+1}-1}(Q_n)\right.\\
~~\left.+{m_{t,t+1}-1}\right]\\
=n+1-\left[ex_{m_{t,t+1}}(Q_n)-ex_{m_{t,t+1}-1}(Q_n)\right]\\
=n+1-\left[2(\lceil n/2\rceil-3)+2\right]\\
=n+1-(n-3)
=4>0.
$
\begin{figure}[!t]
	\flushleft
\includegraphics[width=1\linewidth]{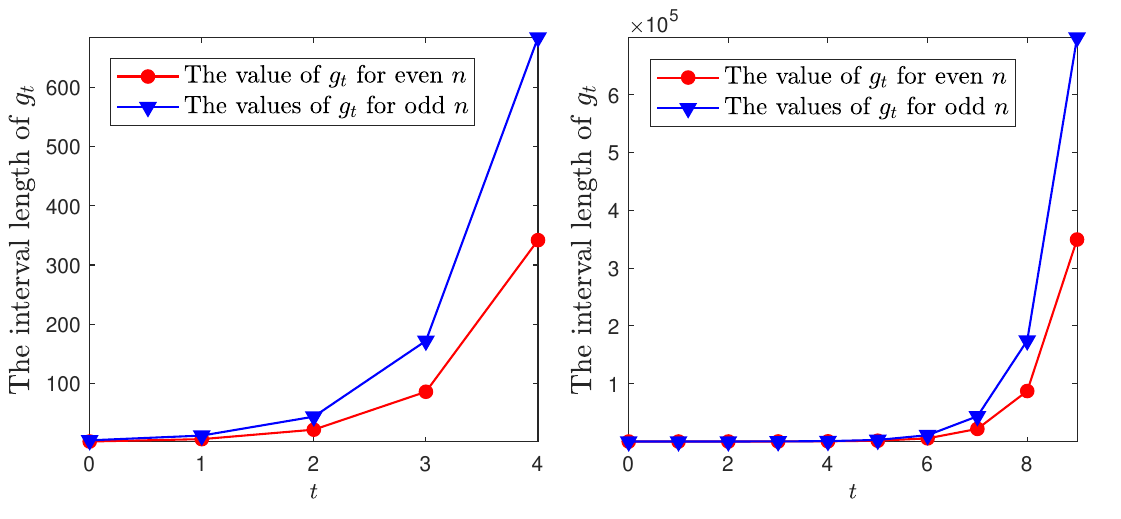}
		\caption{The interval lengths $g_t$ of the concentration phenomena for $t=0$ to $9$.}
\label{figgt}
\end{figure}

Therefore, $\xi_{m_{t,t+1}}(H_n^4)>\xi_{m_{t,t+1}-1}(H_n^4)$. By Lemma \ref{lem4.5}, $\xi_{m_{t,t+1}}(H_n^4)<\xi_{m_{t,t+1}+1}(H_n^4)$.
Thus, $\lambda_{m_{t,t+1}-1}(H_n^4)<\lambda_{m_{t,t+1}}(H_n^4)=\lambda_{m_{t,t+1}+1}(H_n^4)$.
Hence, the lower bound of $h$ is sharp for each $t=0,1,\cdots,\lfloor n/2\rfloor-3,\lfloor n/2\rfloor-1$.

(2) For $t=0,1,2,\cdots,\lfloor n/2\rfloor-3$, $m_{t,0}=2^{\lceil n/2\rceil+t}$ and $m_{t,0}-1=2^{\lceil n/2\rceil+t}-1$.
By Lemma \ref{lem4.4}, $\xi_h(H_n^4)>\xi_{m_{t,0}}(H_n^4)$ for $m_{t,0}<h\leq 2^{n-1}$.
By Lemma \ref{lem4.5}, $\xi_{m_{t,0}-1}(H_n^4)>\xi_{m_{t,0}}(H_n^4)$.
Thus, $\lambda_{m_{t,0}-1}(H_n^4)>\lambda_{m_{t,0}}(H_n^4)=\lambda_{m_{t,0}+1}(H_n^4)$.
Since $m_{\lfloor n/2\rfloor-1,0}=2^{n-1}=|V(H_n^4)|/2$, it can be seen that the upper bound is sharp for $t=\lfloor n/2\rfloor-1$.
Consequently, the remark holds.

Now, the data and images are utilized to convey the findings more intuitively.
According to Formulas (\ref{equa1}) and (\ref{equa2}), Theorems \ref{thm1}, \ref{thm2} and \ref{thmex_m(H_n^4)=f(m)}, the calculated values of $\xi_h(H_n^4)$ and $\lambda_h(H_n^4)$ within the range $3 \leq n \leq 7$ and $1 \leq h \leq 2^{n-1}$ are displayed in Table \ref{tab1}.
The normalized values of $\xi_h(H_n^4) $ and $\lambda_h(H_n^4) $ for $6\leq n\leq10$ are shown in Fig.~\ref{fignormalized values}. From the image, we can visually observe that the function $\xi_h(H_n^4) $ exhibits symmetry and self-similarity, and $\lambda_h(H_n^4)$   exhibits some  concentration behavior for some exponentially large enough $h$ on the interval length of $g_t$,  for $t=0,1, \cdots,\lfloor n / 2\rfloor -1$, $g_t=\lceil\left(2^{2 t+2+\gamma}+2\right) / 3\rceil$.
Furthermore, for $t=0$ to $t=9$,  the values of $g_t$ calculated in this paper are displayed in Fig. \ref{figgt}. These values precisely represent the lengths of all intervals of  concentration   phenomena, which increase exponentially with the increase of $t$.
\subsection{Other four link fault patterns}
In this subsection, we explore $\lambda(\mathscr{P}_i^l,~H_n^4 )$ for $1\leq i\leq5$, $2\leq l\leq n-1$ and $n\geq 2$ by means of $\xi_h(H_n^4)$. Firstly, the following three lemmas offer an upper bound and a lower bound of $\lambda(\mathscr{P}_i^l,~H_n^4 )$ for $1\leq i\leq4$.
\begin{table*}[!t]
\caption{The Exact Values of $\xi_h(H_n^4)$ and $\lambda_h(H_n^4)$ with $3\leq n\leq 7$ and $1\leq h\leq 2^{n-1}$}
\label{tab1}
\begin{center}
\begin{tabular}{c c c c c c c c c c c c c c c c c c c c c c c}
\toprule
$h$  &1&2&3&4&5&6&7&8&9&10&11&12&13&14&15&16&17&18&19&20&21&22\\
\hline
$\xi_h(H_3^4)$ & 4&6&6&4\\
$\lambda_h(H_3^4)$ &4&4&4&4\\
$\xi_h(H_4^4)$ & 5  &   8   &  9  &   8  &  11  &  12  &  11  &   8\\
$\lambda_h(H_4^4)$  & 5  &   8 &    8   &  8   &  8  &   8  &   8  &   8\\
$\xi_h(H_5^4)$ &6  &  10   & 12 &   12  &  16   & 18&    18  &  16  &  20 &   22& 22 &   20 &   22 &   22 &   20 &   16\\
$\lambda_h(H_5^4)$ & 6  &  10  &  12  &  12  &  16  &  16  &  16  &  16  &  16  &  16 &16   & 16  &  16   & 16  &  16  &  16\\
$\xi_h(H_6^4)$ & 7  &  12   & 15  &  16   & 21  &  24  &  25  &  24   & 29  &  32 & 33 &   32  &  35 &   36  &  35   & 32 &   37  &  40  &  41  &  40  &43 &   44\\
$\lambda_h(H_6^4)$ & 7  &  12  &  15&    16 &   21 &   24   & 24  &  24  &  29  &  32 &32  &  32  &  32  &  32  &  32 &   32&    32 &   32  &  32   & 32     &32   & 32 \\
 $\xi_h(H_7^4)$ &8   & 14  &  18   & 20 &   26  &  30   & 32   & 32  &  38 &   42   & 44 &   44 &   48    &50  &  50  &  48  &  54   & 58  &  60   & 60  &  64 & 66\\
$\lambda_h(H_7^4)$ &  8    &14  &  18  &  20 &   26 &   30   & 32 &   32  &  38  &  42   & 44  &  44   & 48   & 48  &  48 &   48   & 54    & 58  &  60  &  60  &  64 & 64  \\
\hline
$h$& 23 &24& 25 &26 &27 &28 &29& 30 & 31 &32& 33& 34& 35& 36& 37& 38& 39& 40& 41& 42& 43& 44\\
\hline
$\xi_h(H_6^4)$& 43   & 40   & 43  &  44 &   43 &   40 &   41  &  40 &   37  &  32\\
$\lambda_h(H_6^4)$ & 32    &32  &  32  &  32   & 32    &32 &   32 &   32  &  32  &  32\\
$\xi_h(H_7^4)$ & 66  &  64  &  68   & 70&    70 &   68  &  70  &  70   & 68  &  64  &  70   & 74 &   76  &  76  &  80 &   82  &  82 &   80   & 84 &   86 & 86 &   84\\
$\lambda_h(H_7^4)$ & 64 &   64 &   64  &  64   & 64 &   64 &   64  &  64  &  64 &   64    &64  &  64 &   64  &  64  &  64   & 64  &  64   & 64 &   64  &  64   &  64&    64 \\
\hline
$h$ &45& 46& 47& 48 &49 &50 &51 &52 &53 &54& 55 &56 &57 &58& 59& 60 &61& 62& 63 &64 \\
\hline
$\xi_h(H_7^4)$   &  86  &  86   & 84  &  80  &  84  &  86  &  86 &   84  &  86   & 86    &84   & 80   & 82  &  82   & 80 &   76  &  76 &   74 &   70 & 64 \\
$\lambda_h(H_7^4)$ &  64 &   64   & 64  &  64  &  64  &  64  &  64&    64   & 64&    64  &  64  &  64  &  64  &  64 &   64   & 64   & 64  &  64 &   64&   64   \\
\bottomrule
\end{tabular}
\end{center}
\end{table*}
\begin{table*}[!t]
\caption{The Exact Values of $\lambda(\mathscr{P}_i^l,~H_3^4)$  for $3\leq n\leq 7$ and $1\leq i\leq 4$}
\label{tab2}
\begin{center}
\begin{tabular}{c c c c c c}
\toprule
$l$& 2 &3 &4&5&6\\
\hline
$\lambda(\mathscr{P}_i^l,~H_3^4)$ &$(3-2)\times2^2=4$ \\
$\lambda(\mathscr{P}_i^l,~H_4^4)$ & $(4-2)\times2^2=8$ &$(4-3)\times2^3=8$  \\
$\lambda(\mathscr{P}_i^l,~H_5^4)$ &$(5-2)\times2^2=12$ &$(5-3)\times2^3=16$ &$(5-4)\times2^4=16$ \\
$\lambda(\mathscr{P}_i^l,~H_6^4)$ &$(6-2)\times2^2=16$ &$(6-3)\times2^3=24$ & $(6-4)\times2^4=32$ & $(6-5)\times2^5=32$ \\
$\lambda(\mathscr{P}_i^l,~H_7^4)$& $(7-2)\times2^2=20$ &$(7-3)\times2^3=32$ &$(7-4)\times2^4=48$  & $(7-5)\times2^5=64$ & $(7-6)\times2^6=64$\\
\bottomrule
\end{tabular}
\end{center}
\end{table*}
\begin{lem}\label{lem6.1}
Let $n,\,l,\,i$ be integers, with $n\geq 2$, $2\leq l\leq n-1$, and $1\leq i\leq4$.
Then $\lambda(\mathscr{P}_i^l,~H_n^4 )\leq \xi_{2^l}(H_n^4)=(n-l)2^l$.
\end{lem}
\begin{proof}
Since
 $$H_n^4\left[L_{2^l}^n\right]=H_n^4[\underbrace{00 \cdots 00}_{n-l} X_l X_{l-1} \cdots X_2X_1]\cong H_l^4,$$
then $H_n^4\left[L_{2^l}^n\right]$ is connected.
It can be noticed that both the minimum degree and the average degree of $H_n^4[L_{2^l}^n]$ are $l+1$, for $H_l^4$ is a $(l+1)$-regular graph.
Owing to $|V(H_n^4[L_{2^l}^n])|=2^l$ and   $2\leq l\leq n-1$, $|V\left(H_n^4\left[\overline{L_{2^l}^n}\right]\right)|=2^n-2^l>2^l$.

For another thing, we can get that
  $H_n^4\left[\overline{L_{2^l}^n}\right]=\\H_n^4[\underbrace{00 \cdots 0}_{n-l-1} 1 X_l X_{l-1} \cdots X_1 \bigcup \underbrace{00 \cdots 0}_{n-l-2} 1 X_{l+1} X_l \cdots X_1 \bigcup \cdots \\\bigcup 01 X_{n-2} X_{n-3} \cdots X_1 \bigcup 1 X_{n-1} X_{n-2} \cdots X_1]$.
For each $l \leq$ $\beta \leq n-1, H_n^4[\underbrace{00 \cdots 0}_{n-\beta-1} 1 X_{\beta} X_{\beta-1} \cdots X_1]\cong H_{\beta}^4$.
Then $H_n^4\left[\overline{L_{2^l}^n}\right]$ is connected.
Not only the minimum degree but also the average degree of $H_n^4\left[\overline{L_{2^l}^n}\right]$ is at least $l+1$, for the reason that $H_{\beta}^4$ is a $(\beta+1)$-regular graph.
And thirdly, from the Lemma \ref{lem2.4},
deleting the edge-cut $[H_{2^l}^n,\overline{H_{2^l}^n}]$ of $H_n^4$ only produces exactly two components and both $H_n^4[L_{2^l}^n]$ and $H_n^4[\overline{H_{2^l}^n}]$ satisfy the properties of $\mathscr{P}_i^l$ for $1\leq i\leq4$. This yields that $[H_{2^l}^n,\overline{H_{2^l}^n}]$ is a $\mathscr{P}_i^l$-conditional edge-cut of $H_n^4$.
Hence, $\xi_{2^l}(H_n^4)=(n+1)2^l-2|E(H_n^4[L_{2^l}^n])|=(n+1)2^l-(l+1)2^l=(n-l)2^l$.
By the definition of $\lambda(\mathscr{P}_i^l,~H_n^4)$, we have $\lambda(\mathscr{P}_i^l,~H_n^4)\leq \xi_{2^l}(H_n^4)=(n-l)2^l$ for $1\leq i\leq4$.
\end{proof}
To provide a lower bound of $\lambda(\mathscr{P}_i^l,~H_n^4 )$ for $1\leq i\leq4$, we require the following lemma, which relates the number of vertices and the average degree of the induced subgraph of $H_n^4$.

\begin{lem}\label{lem6.2}
Let $I$ be a subset of $V(H_n^4)$. If the average degree of the induced subgraph $H_n^4[I]$ is $l$, then $I$ contains at least $2^{l-1}$ vertices.
\end{lem}
\begin{proof}
We deduce, by induction on $|I|$.
This is  true for $|I|=2$. Suppose, then, that $|I|\geq 3$.
Assume that  $I\subset V\left(H_n^4\right)$, $|I|=k$ and the average degree of the induced subgraph $H_n^4[I]$ is $l$.
Let the vertex $y_{n}y_{n-1}\cdots y_2y_1\in V\left({H_n^4}^{i(0)}\right)$ or $V\left({H_n^4}^{i(1)}\right)$ if and only if the $i$-bit $y_i=0$ or $y_i=1$.
If $H_n^4={H_n^4}^{i(0)}\bigcup {H_n^4}^{i(1)}$, then it can be asserted that $H_n^4$ can be decomposed into two subgraphs on the $i$-bit for $1\leq i\leq n$.
 Two subgraphs ${H_n^4}^{i(0)}$ and ${H_n^4}^{i(1)}$ may composed into $H_n^4$ for $1\leq i\leq n$ and  $V\left(H_n^4\right)=V\left({H_n^4}^{i(0)}\right)\bigcap V\left({H_n^4}^{i(1)}\right)$.
Let $I_\alpha=I\cap V\left({H_n^4}^{i(\alpha)}\right)$ with $|I_\alpha|=k_\alpha$, for $\alpha=0,~1$.
Assume that $k_0\leq k_1 $ without loss of generality.
Let there exist $m$ edges of $H_n^4[I]$ between $I_0$ and $I_1$.
For each vertex in $I_0$, there are at most $2$ edges and $1$ edge are incident with $V(I_1)$ for $i=1,~2$ and $i=3$, respectively. Thus, $m\leq 2k_0$ and $m\leq k_0$, respectively.
For $i=1,\,2$, we have
$k(\log_2k+1)\geq k_0\cdot\log_2k_0+k_1\cdot\log_2k_1+2\cdot 2k_0
\geq
k_0\cdot\log_2k_0+k_1\cdot\log_2k_1+2m
\geq
\sum_{v\in I}d_{H_n^4}(v)=kl$.
For $3\leq i\leq n$, thus
$k(\log_2k+1)\geq
k_0\cdot(\log_2k_0+1)+k_1\cdot(\log_2k_1+1)+2\cdot k_0
\geq
k_0\cdot(\log_2k_0+1)+k_1\cdot(\log_2k_1+1)+2m
\geq \sum_{v\in I}d_{H_n^4}(v)=kl$.
 Above all, $k\left(\log_2k+1\right)\geq kl$.
 It can be checked that $k\geq 2^{l-1}$.
 The proof is completed.
\end{proof}

\begin{lem}\label{lem6.3}
Let $n,\,l,\,i$ be three integers with $n\geq 2$, $2\leq l\leq n-1$, and $1\leq i\leq4$.
Then $\lambda(\mathscr{P}_i^l,~H_n^4 )\geq \xi_{2^l}(H_n^4)=(n-l)2^l$.
\end{lem}
\begin{proof}
Considering the given information, assume that  $F$ is any minimum $\mathscr{P}_i^l$-conditional edge-cut of $H_n^4$ and $|F|=\lambda(\mathscr{P}_i^l,~H_n^4 )$.
By  Lemma \ref{lem2.4}, deleting $F$  produces exactly two components and each component possesses $\mathscr{P}_i^l$.
Let   $S$ stand for the smaller component.
If  $\delta(S)\geq l+1$, then $d(S)\geq l+1$. This is clearly true $|V(S)|\geq 2^l$ by Lemma~\ref{lem6.2}.
Moreover, by Formula~(\ref{equa1}) and Lemma~\ref{lem4.4}, $|F|=|[V(S),\overline{V(S)}]_{H_n^4}|\geq \xi_{|V(S)|}(H_n^4)\geq \xi_{2^l}(H_n^4)=(n-l)2^l$.
We thus have $\lambda(\mathscr{P}_i^l,~H_n^4)\geq \xi_{2^l}(H_n^4)=(n-l)2^l$ for $n\geq 2$, $2\leq l\leq n-1$ and $1\leq i\leq4$, the inequality following from Formula (\ref{equa2}).
\end{proof}

Now we are ready to prove the Theorem~\ref{thm3}.
\\
\textbf{The proof of Theorem~\ref{thm3}}\quad
 By the Lemma~\ref{lem6.1} and Lemma~\ref{lem6.3}, we firstly can have $\lambda(\mathscr{P}_i^l,~H_n^4)=\xi_{2^l}(H_n^4)=(n-l)2^l$ holds for $1\leq i\leq4$.

Next, we consider the cyclic edge-connectivity $\lambda(\mathscr{P}_5^l,H_n^4)$.
For $n\geq 3$, assume $F$ is  any cyclic edge-cut of $H_n^4$. This suggests that $G-F$ yields two components and  let $S$ be the smaller one.
 Thus $|V(S)|\geq3$ in virtue of $S$ contains a cycle.
 Note that $\xi_m(H_n^4)$ is larger than $\xi_3(H_n^4)$ and $\xi_4(H_n^4)$ for $n\geq3 $ and $5\leq m\leq 2^{n-1}$.
It can be  obtained that  $|F|=|[V(S),\overline{V(S)}]_{H_n^4}|\geq \xi_{|V(S)|}(H_n^4)\geq \min\{\xi_3(H_n^4),~\xi_4(H_n^4)\}$.
Also $\xi_4(H_n^4)=4n-8\geq \xi_3(H_n^4)=3n-3$ for $n\geq5$.
Moreover, $[L_3^n,\overline{L_3^n}]$ and $[L_4^n,\overline{L_4^n}]$ are cyclic edge-cuts of $H_n^4$ and
the cardinality of $[L_3^n,\overline{L_3^n}]$ and $[L_4^n,\overline{L_4^n}]$ are $\xi_3(H_n^4)$ and $\xi_4(H_n^4)$, respectively.
Given the above, it is now  checked that
$\lambda(\mathscr{P}_5^l,~H_n^4)=
 4n-8$ if $n=3,~4$ and
 $\lambda(\mathscr{P}_5^l,~H_n^4)=3n-3$ if $n\geq 5.$
$\hfill\blacksquare$
\begin{rem}
For $n\geq 2,\,0\leq l\leq 1$, it is supplemented that $\lambda^l(H_n^4)=\overline{\lambda^l}(H_n^4)=\lambda_{2^l}(H_n^4)=(n+1-l)2^l$.
\end{rem}

By Theorem \ref{thm3}, the exact values of
$\lambda\left(\mathscr{P}_i^l,~H_n^4\right)$   are showed in Table \ref{tab2} within the range $3 \leq n \leq 7$ and $1\leq i\leq4$.
These data and images effectively diagnose the reliability of generalized $K_4$-Hypercubes.

\section{Conclusion}\label{sec6}
%One method to diagnose the stability of a multiprocessor system is to examine the structure of the system under many link fault patterns.
This paper provides the exponential edge fault-tolerance and  explores the reliability  of the class of  generalized $K_4$-hypercubes based on five link fault patterns, which are the $h$-extra edge-connectivity $\lambda_h$, the $l$-edge-connectivity $\lambda^l$, the $l$-average degree edge-connectivity $\overline{\lambda^l}$, the $l$ -embedded edge-connectivity $\eta_l$, and  the cyclic edge-connectivity $\lambda_c$.
We present the precise values of  $\lambda_h(H_n^4)$ for   $1\leq h\leq 2^{\lceil n/2  \rceil }$ and
 $\lambda_h(H_n^4)$ demonstrates $\lfloor n/2\rfloor-1$ concentration phenomena  for  $2^{\lceil n/2\rceil+t}-g_t \le h\le2^{\lceil n/2\rceil+t}$, $0\leq t \leq\lfloor n/2\rfloor-1 $, $g_t=\lceil(2^{2t+2+\gamma})/3\rceil$.
 %$\gamma=0$ if $n$  is even and $\gamma=1$ if $n$ is odd.
Moreover,  $\lambda^l(H_n^4)$, $\overline{\lambda^l}(H_n^4)$, $\lambda_{2^l}(H_n^4)$ and $\eta_l(H_n^4)$ share a common value $(n-l)2^l$ for $n\geq3$, $2\leq l\leq n-1$ and $ 1\leq i\leq4$ and  determine that $\lambda_c(H_n^4)$ for $n\geq3$.
Specifically, due to the enhanced hypercube $Q_{n,n-1}\in \mathscr{H}_n^4$, all the conclusions in this paper apply to $Q_{n,n-1}$ as well.
In the future, we may also attempt to design an algorithm to solve the precise values of the remaining  intervals of $\lambda_h(H_n^4)$.

The findings of the article demonstrate that the optimal solution of the edge isoperimetric problem can be applied to explore several conditional edge-connectivities and can measure the reliability and fault tolerance of interconnection networks.
Afterwards, we can further study that as $n$ tends to infinity, the function  $\xi_h(H_n^4)$ is continuous everywhere but not differentiable and has a fractal structure.
By converting the optimal solution   of the edge isoperimetric problem into the construction of a continuous everywhere but not differentiable function, we could establish a link between seemingly unrelated graph theory and special continuous functions.
Similarly, this method can be applied to evaluate various conditional edge-connectivities in other interconnection networks.
%\bibliographystyle{IEEEtranS}
%\bibliography{IEEEabrv,reference}
% Generated by IEEEtranS.bst, version: 1.13 (2008/09/30)

 \vspace{0.3cm}
\begin{IEEEbiography}[{\includegraphics[width=1in,height=1.25in,clip,keepaspectratio]{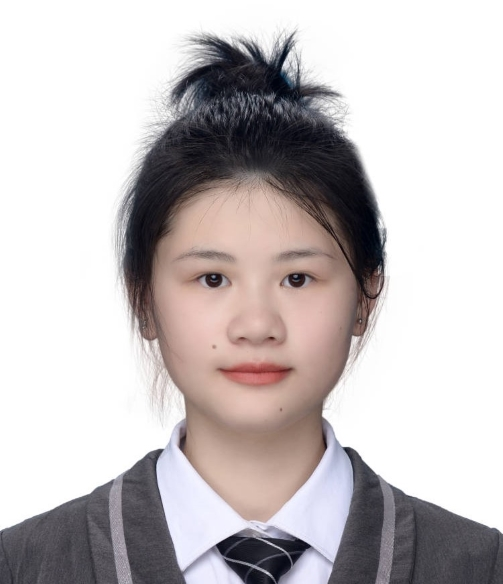}}]{Shuqian Cheng}
received the B.S. degree in mathematics and applied mathematics from China West Normal University, Nanchong, China, in 2022. Currently, she is  pursuing the M.S. degree in mathematics at Xinjiang University, Urumqi, China. Her research interests include
graph theory and its applications,
reliability analysis of interconnection networks, data center networks, and design and analysis of algorithms.
\end{IEEEbiography}
\vspace{-1.5cm}
\begin{IEEEbiography}[{\includegraphics[width=1in,height=1.25in,clip,keepaspectratio]{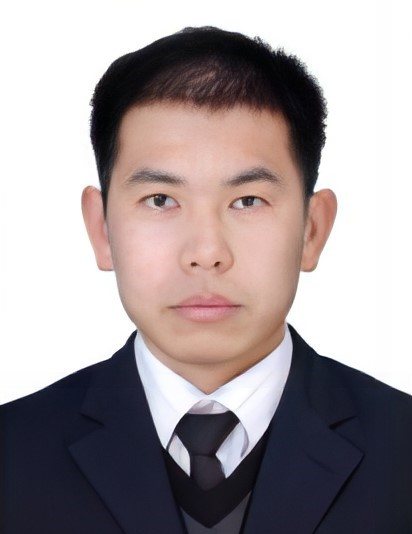}}]{Mingzu Zhang}
received the B.S. degree in mathematics from the Tianjin University of Technology and Education, Tianjin, China, in 2010, the M.S. degree in mathematics from Xinjiang University, Urumqi, China, in 2014, and the Ph.D. degree in applied mathematics from the School of Mathematical Science, Xiamen University, Xiamen, China, in 2018.
Currently, he is an Associate Professor with the Department of Mathematics at Xinjiang University, Urumqi, China. He has published more than  thirty  papers in journals \textsl{IEEE Transactions on Reliability}, \textsl{Information Sciences}, \textsl{Journal of Combinatorial Optimization}, \textsl{Discrete Applied Mathematics}, \textsl{Theoretical Computer Science} and \textsl{Journal of Parallel and Distributed Computing}.
His main research interests include
data center networks,
network optimization, fractal graph theory, and algorithms.
\end{IEEEbiography}
%\vspace{-3cm}

\begin{IEEEbiography}[{\includegraphics[width=1in,height=1.25in,clip,keepaspectratio]{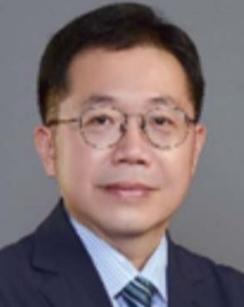}}]{Sun-Yuan Hsieh}(Fellow, IEEE) received the Ph.D. degree in computer science from National Taiwan University, Taipei, Taiwan, in 1998. Then, he served the compulsory military two years. He held the position of Assistant Professor with the Department of Computer Science and Information Engineering at National Chi Nan University from 2000 to 2002. In February 2002, he was appointed as a member of  the Department of Computer Science and Information Engineering, National Cheng Kung University.
 Dr. Hsieh awards include the 2020 ACM Distinguished Scientist,
 the  2019 Kwoh Ting Li Honorable Scholar Award,
 the 2016 Outstanding Research Award of Taiwan Ministry of Science and Technology,
 the 2014 Outstanding Engineering Professor Award of Chinese Institute of Engineers.
the  2013 Outstanding Electronic Engineering Professor Award of Chinese Institute of Electrical Engineers,
  from the IEEE Tainan Section, in 2011, the  IEEE Outstanding Technical Achievement Award  and so on.

   He is fellow of the Institution of Engineering and Technology and fellow of the British Computer Society.  He has an outstanding record of research achievements in the fields of algorithms and fault-tolerant computing for interconnection networks. His current research interests include data center, graph theory, fault diagnosis, and design and analysis of algorithms.
\end{IEEEbiography}

\begin{IEEEbiography}[{\includegraphics[width=1in,height=1.25in,clip,keepaspectratio]{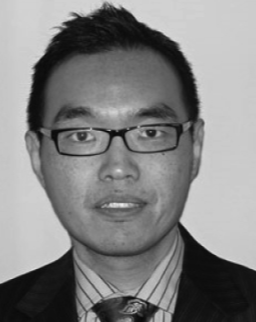}}]{Eddie Cheng}
received the Ph.D. degree in combinatorics and optimization from the University of Waterloo, Canada, in 1995.
He is with the Department of Mathematics and Statistics, Oakland University.
From 2010 to 2013, he was the Chair of the Department  and from January 2016 to April 2016, was the Acting Chair.
 He is an editor for numerous journals, including Networks, International Journal of Machine Learning and Cybernetics, Discrete Applied Mathematics, Journal of Interconnection Networks, International Journal of Computer Mathematics: Computer Systems Theory, and International Journal of Parallel, Emergent and Distributed Systems.
 His main research interests include interconnected network, fault diagnosis, binary string, and breadth-first search.
\end{IEEEbiography}
\end{document}